\documentclass[10pt, a4paper]{amsart}
\pdfoutput=1
\usepackage{amsmath}
\usepackage{amscd}
\usepackage{amssymb}
\usepackage[all]{xy}
\usepackage[T1]{fontenc}
\usepackage{lmodern}
\usepackage{graphicx}
\usepackage{tikz-cd}
\usepackage[colorlinks,pagebackref=true]{hyperref}
\hypersetup{colorlinks,linkcolor={blue},citecolor={blue},urlcolor={red}}

\newtheorem*{theorem*}{Theorem}
\newtheorem*{proposition*}{Proposition}
\newtheorem{theorem}[equation]{Theorem}
\newtheorem{corollary}[equation]{Corollary}
\newtheorem{proposition}[equation]{Proposition}
\newtheorem{lemma}[equation]{Lemma}
\theoremstyle{definition}
\newtheorem{definition}[equation]{Definition}
\newtheorem{remark}[equation]{Remark}
\newtheorem{example}[equation]{Example}

\newtheorem{question}[equation]{Question}

\newtheorem{convention}[equation]{Convention}

\newtheorem{notation}[equation]{Notation}
\numberwithin{equation}{section}

\newcommand{\opn}{\operatorname}
 
\newcommand{\cat}[1]{\operatorname{\mathsf{#1}}}

\newcommand{\rmitem}[1]{\item[\text{\textup{(#1)}}]}
\newcommand{\mfrak}[1]{\mathfrak{#1}}

\newcommand{\mrm}[1]{\mathrm{#1}}

\newcommand{\Ga}{\Gamma}
\newcommand{\Ja}{\mathrm{Q}}
\newcommand{\La}{\Lambda}
\newcommand{\si}{\sigma}

\renewcommand{\i}{\mfrak{i}}

\renewcommand{\a}{\mfrak{a}}

\renewcommand{\aa}{\bsym{a}}

\newcommand{\K}{\mathbb{K}}

\newcommand{\tup}[1]{\textup{#1}}
\newcommand{\bsym}[1]{\boldsymbol{#1}}

\newcommand{\ot}{\otimes}

\title[Weakly stable torsion classes]{Weakly stable torsion classes}

\author{Rishi Vyas}
\email{vyas.rishi@gmail.com}

\thanks{This work was completed when the author was supported by: Israel Science Foundation grants  253/13 and 170/12, the Center for Advanced Studies in Mathematics at Ben-Gurion University of the Negev, and the Israel Council for Higher Education.}

\begin{document}

\begin{abstract}
Weakly stable torsion classes were introduced by the author and Yekutieli to provide a torsion theoretic characterisation of the notion of weak proregularity from commutative algebra. In this paper we investigate weakly stable torsion classes, with a focus on aspects related to localisation and completion. We characterise when torsion classes arising from left denominator sets and idempotent ideals are weakly stable. We show that every weakly stable torsion class $\cat{T}$ can be associated with a dg ring $A_{\cat{T}}$; in well behaved situations there is a homological epimorphism $A\to A_{\cat{T}}$. We end by studying torsion and completion with respect to a single regular and normal element.  
\end{abstract}

\maketitle

\section{Introduction} \label{s:01}

Weak proregularity was first defined by Alonso, Jeremias, and Lipman in \cite{AJL} (the phenomenon that it formalises, however, was already observed by Grothendieck in \cite{LC}). It was  further studied by Schenzel in \cite{Sch}, Porta, Shaul and Yekutieli in \cite{PSY1} and \cite{PSY2}, and Positselski \cite{Po}.

 Let $\a$ be an ideal in a commutative ring $A$. There are two functors associated with $\a$: the $\a$-torsion functor $\Gamma_{\a}$ and the $\a$-adic completion functor $\Lambda_{\a}$. The importance of weak proregularity in commutative algebra arises from the fact that it is the natural context in which to study a collection of equivalences and dualities that relate $\Gamma_{\a}$ and $\Lambda_{\a}$. Let $\cat{T}_{\a}$ denote the class of all $A$-modules $M$ such that $\Ga_{\a}(M)=M$. If $A$ is a ring we denote the category of left $A$-modules by $\cat{M}(A)$, and use $\cat{D}(A)$ for the corresponding derived category.  Consider the following theorem.

\begin{theorem*} \cite{PSY1}
Let $\a$ be a weakly proregular ideal in a commutative ring $A$. 
\begin{enumerate}
\rmitem {i}The functor $\mrm{L} \La_{\a}: \cat{D}(A)\to \cat{D}(A)$ is right adjoint to $\mrm{R} \Ga_{\a}: \cat{D}(A)\to \cat{D}(A)$.
\rmitem{ii} The functors $\mrm{R} \Ga_{\a}$ and $\mrm{L} \La_{\a}$ are idempotent. 
\rmitem{iii}Let $\cat{D}(A)_{\a \! \tup{-tor}}$ and $\cat{D}(A)_{\a \! \tup{-com}}$ be the essential images of $\mrm{R}\Ga_{\a}$ and $\mrm{L}\Lambda_{\a}$, respectively. Then, the functor 
\[ \mrm{R} \Ga_{\a} : \cat{D}(A)_{\a \! \tup{-com}} \to \cat{D}(A)_{\a \! \tup{-tor}} \]
is an equivalence of categories, with quasi-inverse $\mrm{L}\Lambda_{\a}$. 
\rmitem{iv} $\cat{D}(A)_{\a \! \tup{-tor}}$ is the full subcategory of objects $M\in \cat{D}(A)$ such that  $\opn{H}^{i}(M)\in \cat{T}_{\a}$ for all $i \in \mathbb{Z}$.
\end{enumerate}
\end{theorem*}

The statement given above is excerpted from \cite{PSY1} (where it is called the \emph{MGM equivalence}), but the result, in content, is the denouement of a large body of work, including that of Grothendieck \cite{LC}, Matlis \cite{Ma}, Greenlees and May \cite{GM}, Alonso, Jeremias and Lipman \cite{AJL}, Schenzel \cite{Sch}, and Porta, Shaul, and Yekutieli \cite{PSY1}. Weakly proregular ideals are ubiquitous: every ideal in a commutative noetherian ring is weakly proregular.

Traditionally, weak proregularity of an ideal $\a$ was defined via properties of Koszul complexes associated to explicit finite generating sets of $\a$. Such definitions do not generalise well to noncommutative ring theory. In \cite{VY}, the author and Yekutieli characterised when $\a$ is weakly proregular via a categorical property of the functor $\Ga_{\a}$ and associated torsion class $\cat{T}_{\a}$. This property was formalised in the definition of a \emph{weakly stable} torsion class (Definition \ref{d:02}).

\begin{theorem*} \cite[Theorem 4.12]{VY}
Let $A$ be a commutative ring, let $\aa$ be a finite sequence of elements of $A$, let $\a$ be the ideal generated by $\aa$, and let $\cat{T}_{\a}$ be the associated torsion class in $\cat{M}(A)$. The following two conditions are equivalent\tup{:} 
\begin{enumerate}
\rmitem{i} The sequence $\aa$ is weakly proregular. 
\rmitem{ii} The torsion class $\cat{T}_{\a}$ is weakly stable.
\end{enumerate}
\end{theorem*}

Let $A$ be a ring and $\cat{T}$ a torsion class in $\cat{M}(A)$, with associated torsion functor $\Ga_{\cat{T}}$. There are various conditions we might want $\cat{T}$ to satisfy to ensure good behaviour: these include \emph{finite dimensionality} (Definition \ref{d:05}) and \emph{quasi-compactness} (Definition \ref{d:08}). This leads to a  noncommutative version of the MGM equivalence.

\begin{theorem*} \cite[Proposition 3.7, Theorem 8.3]{VY}
Let $A$ be a ring, flat over a commutative base ring $\K$. Let $\cat{T}$ be a quasi-compact, finite dimensional, weakly stable torsion class in $\cat{M}(A)$. 
\begin{enumerate}
\rmitem{i} The functor $\mrm{R}\Ga_{\cat{T}}:\cat{D}(A)\to \cat{D}(A)$ admits a right adjoint $\mathrm{G}_{\cat{T}}:\cat{D}(A)\to \cat{D}(A)$.
\rmitem{ii} The functors $\mrm{R}\Ga_{\cat{T}}$ and $\mathrm{G}_{\cat{T}}$ are idempotent.
\rmitem{iii} Let $\cat{D}(A)_{\cat{T} \! \tup{-tor}}$ and $\cat{D}(A)_{\cat{T} \! \tup{-com}}$ be the essential images of $\mrm{R}\Ga_{\cat{T}}$ and $\mathrm{G}_{\cat{T}}$, respectively. Then, the functor 
\[ \mrm{R} \Ga_{\cat{T}} : \cat{D}(A)_{\cat{T} \! \tup{-com}} \to \cat{D}(A)_{\cat{T} \! \tup{-tor}} \]
is an equivalence of categories, with quasi-inverse $G_{\cat{T}}$. 
\rmitem{iv} $\cat{D}(A)_{\cat{T} \! \tup{-tor}}$ is the full subcategory of objects $M\in \cat{D}(A)$ such that  $\opn{H}^{i}(M)\in \cat{T}$ for all $i \in \mathbb{Z}$.
\end{enumerate}
\end{theorem*}

In this paper, we further explore properties of weakly stable torsion classes. Section \ref{s:02} is a short overview of some basic definitions and elementary properties of functors that will be required. 

In Section \ref{s:07}, we give a quick primer on torsion classes in module categories, introducing relevant notation and recalling definitions that will be required later. In particular, we recall the definitions of weakly stable, finite dimensional, and quasi-compact torsion classes from \cite{VY}. We discuss certain aspects of weakly stable torsion classes and their associated torsion functors.

Section \ref{s:03} is concerned with constructing a stock of examples. Every left denominator set $S$ in a ring $A$ can be associated with a torsion class $\cat{T}_{S}$ and torsion functor $\Ga_{S}$. $\cat{T}_{S}$ is the class of all modules which are element-wise annihilated by some element of $S$; if $M$ is a module, then $\Ga_{S}(M) = \{m\in M\ \vert \ sm=0 \ \text{for some}\ s\in S\}.$ The following theorem is excerpted from Theorem \ref{t:02}.

\begin{theorem*}
Let $S$ be a left denominator set in a ring $A$. Then, the following are equivalent:
\begin{enumerate}
\rmitem{i} The torsion class $\cat{T}_{S}$ is weakly stable.
\rmitem{ii} For any injective module $E$, the canonical localisation map $E\to E_{S}$ is surjective.
\rmitem{iii} $\mrm{R}\Gamma_{S} \cong (A\to A_{S})\ot_{A}^{\mrm{L}}(-)$, where $A\to A_{S}$ is the complex with $A$ in degree $0$, $A_{S}$ in degree $1$, and the canonical localisation map as differential.
\rmitem{iv} $\Gamma_{S}$ has right cohomological dimension less than or equal to one.
\end{enumerate}
\end{theorem*}

From the above theorem, it follows that when $\Ga_{S}$ is weakly stable, it is also quasi-compact and finite dimensional. The same is true for ideals in a commutative ring: if $\a$ is a weakly proregular ideal in a commutative ring $A$, $\Ga_{\a}$ is automatically quasi-compact and finite dimensional (\cite[Corollary 4.21]{VY}). In the remainder of Section \ref{s:03}, we construct a sequence of examples to show that such nice behaviour is not the norm. 

There is a process by which every torsion class $\cat{T}$ in $\cat{M}(A)$ can be associated with a functor $\Ja_{\cat{T}}:\cat{M}(A)\to \cat{M}(A)$ and a full subcategory $\cat{M}(A)_{\cat{T}}$ of $\cat{M}(A)$. The category $\cat{M}(A)_{\cat{T}}$ is a localisation of $\cat{M}(A)$, with corresponding localisation functor $\Ja_{\cat{T}}^{2}$ (in Corollary \ref{c:03}, we show that $\Ja_{\cat{T}} \cong \Ja_{\cat{T}}^{2}$ for weakly stable $\cat{T}$). In Section \ref{s:04} we study the derived functor $\mrm{R}\Ja_{\cat{T}}$. The following result appears as Theorem \ref{t:03}.

\begin{theorem*}
Let $A$ be a ring, flat over a commutative base $\K$, and let $\cat{T}$ be a weakly stable torsion class in $\cat{M}(A)$. Then, there exists a dg $\K$-ring $A_{\cat{T}}$, a morphism of dg $\K$-rings $\phi_{\cat{T}}: A\to A_{\cat{T}}$, and a morphism $\psi_{\cat{T}}: A_{\cat{T}}\to \mathrm{R}\Ga_{\cat{T}}(A)[1]$ in $\cat{D}(A\otimes_{\K} A^{\mrm{op}})$ such that the triangle \[\mathrm{R}\Ga_{\cat{T}}(A)\xrightarrow{\si_{A}^{\mrm{R}}} A\xrightarrow{\phi_{\cat{T}}} A_{\cat{T}}\xrightarrow{\psi_{\cat{T}}} \mathrm{R}\Ga_{\cat{T}}(A)[1]\] is distinguished in $\cat{D}(A\otimes_{\K} A^{\mrm{op}})$.
\end{theorem*}

When $\cat{T}$ is quasi-compact, finite dimensional, and weakly stable, the map $\phi_{\cat{T}}$ is a homological epimorphism (Corollary \ref{c:01}), $\mrm{R}\Ja_{\cat{T}} \cong A_{\cat{T}}\ot_{A}^{\mrm{L}} (-)$ (Theorem \ref{p:11}), and the corresponding fully faithful functor $\opn{rest}_{\phi_{\cat{T}}}:\cat{D}(A_{\cat{T}})\to \cat{D}(A)$ has as its essential image all complexes $M\in \cat{D}(A)$ such that $\mrm{R}\Ga_{\cat{T}}(M)=0$ (Corollary \ref{c:02}).

If $\a$ is a weakly proregular ideal in a commutative ring $A$, the right adjoint of $\mrm{R}\Ga_{\a}$ is $\mrm{L}\Lambda_{\a}$. This leads to the following question: for which weakly stable torsion classes $\cat{T}$ does the functor $\mrm{R}\Ga_{\cat{T}}$ admit an adjoint of the form $\mrm{L}\Lambda_{\cat{T}}$, where $\Lambda_{\cat{T}}$ is some functor from $\cat{M}(A)$ to itself? In Section \ref{s:06}, we first provide an example to show that this does not always happen. This is then followed by studying a positive case, that of a single normal and regular element. 

If $s$ is a normal and regular element in a ring $A$, $S=\{s^{n}\}_{n\in \mathbb{N}}$ is a left denominator set. Denote the corresponding torsion functor by $\Ga_{s}$. Let $\Lambda_{s}$ denote the $(s)$-adic completion functor. The following result is a paraphrasing of  Theorem \ref{t:01}. It appears explicitly as Corollary \ref{c:04}.

\begin{theorem*}
Let $s$ be a normal and regular element in a ring $A$. Then, the functor $\mrm{L}\Lambda_{s}:\cat{D}(A)\to \cat{D}(A)$ is right adjoint to $\mrm{R}\Ga_{s}:\cat{D}(A)\to \cat{D}(A)$.
\end{theorem*}

 \subsection*{Acknowledgements}
 The author would like to thank Amnon Yekutieli for his assistance and many suggestions regarding the material in this paper.

\begin{convention} \label{con:01}
Throughout this paper, $\K$ will be a fixed commutative base ring and all constructions will be over $\K$. By ring, we will mean what is ordinarily referred to as a unital, associative $\K$-algebra. Similarly, by differential graded (i.e.~dg)  ring we mean a unital, associative dg $\K$-algebra. Morphisms of (dg) rings will be $\K$-linear. Pre-additive, additive, and abelian categories will all mean categories enriched over $\cat{M}(\K)$, the category of unital $\K$-modules. Additive functors will be $\K$-linear. An unadorned tensor product $\otimes$ will always mean $\otimes_{\K}$.
\end{convention}

\begin{notation}
If $A$ is a ring, we will use $\cat{M}(A)$ for the abelian category of unital left $A$-modules, $\cat{C}(A)$ for the category of complexes of unital left $A$-modules, $\cat{K}(A)$ for its homotopy category, and $\cat{D}(A)$ for the derived category. All complexes are indexed cohomologically. We use $\cat{D}^{+}(A)$ for the full subcategory of complexes in $\cat{D}(A)$ with bounded below cohomology. If $A$ is a dg ring, $\cat{C}(A)$ will be the category of unital left $A$-modules, with homotopy category $\cat{K}(A)$ and derived category $\cat{D}(A)$.

$A^{\mrm{op}}$ is the opposite (dg) ring of $A$. Right $A$-modules are canonically identified with left modules over $A^{\mrm{op}}$. There is also the enveloping algebra of $A$, $A^{\mrm{en}} := A\ot A^{\mrm{op}}$. $A$-bimodules are identified with left $A^{\mrm{en}}$-modules.

 If $M$ is an object of $\cat{C}(A)$, $\opn{H}^{i}(M)$ is the $i^{\mrm{th}}$ cohomology of $M$, and $\opn{Z}^{i}(M)$ is the module of $i^{\mrm{th}}$ cocycles. If $M$ and $N$ are objects of $\cat{C}(A)$, we use  $\mathrm{Hom}_{A}(M,N)$ for the $\K$-linear complex of homomorphisms between $M$ and $N$; in particular, if $A$ is a ring and $M$ and $N$ are in $\cat{M}(A)$, $\mathrm{Hom}_{A}(M,N)$ is the standard $\K$-module of maps. For morphisms in $\cat{C}(A)$, $\cat{K}(A)$, and $\cat{D}(A)$, we will use $\mrm{Hom}_{\cat{C}(A)}(M,N)$, $\mrm{Hom}_{\cat{K}(A)}(M,N)$ and $\mrm{Hom}_{\cat{D}(A)}(M,N)$ respectively; there are isomorphisms $\opn{H}^{0}(\mrm{Hom}_{A}(M,N))\cong \mrm{Hom}_{\cat{K}(A)}(M,N)$ and $\opn{Z}^{0}(\mrm{Hom}_{A}(M,N))\cong \mrm{Hom}_{\cat{C}(A)}(M,N)$. If $M\in \cat{C}(A^{\mrm{op}})$ and $N\in \cat{C}(A)$, $M\ot_{A} N$ is the tensor product of $M$ and $N$.
 
If $M, N\in \cat{D}(A)$, we use $\mrm{RHom}_{A}(M,N)$ for derived $\mrm{Hom}$; $\opn{H}^{0}(\mrm{RHom}_{A}(M,N)) \cong \mrm{Hom}_{\cat{D}(A)}(M,N)$. If $M\in \cat{D}(A^{\mrm{op}})$ and $N\in \cat{D}(A)$, $M\ot^{\mrm{L}}_{A} N$ is the derived tensor product of $M$ and $N$.
$\cat{C}(A)$, $\cat{K}(A)$, and $\cat{D}(A)$ are all equipped with a translation functor $[1]$: $(M[1])^{i} = M^{i+1}$, and $\partial^{i}_{M[1]} = -\partial^{i+1}_{M}$. $[n]:= [1]^{n}$.

\end{notation}

\begin{convention}
Suppose $A$ and $B$ are rings, and $F:\cat{M}(A)\to \cat{M}(B)$ an additive functor. Through termwise application, $F$ extends to a functor $\cat{C}(A)\to \cat{C}(B)$ and thus induces a triangulated functor $\cat{K}(A)\to \cat{K}(B)$. We will denote both of these extensions by $F$ as well. Similarly, if $G:\cat{M}(A)\to \cat{M}(B)$ is another additive functor and  $\alpha:F\to G$ a morphism of functors, then $\alpha$ naturally extends to a morphism from $F$ to $G$ viewed as functors from either $\cat{C}(A)\to \cat{C}(B)$ or $\cat{K}(A)\to \cat{K}(B)$. We will continue to denote these morphisms of functors by $\alpha$. In the second case, $\alpha$ is a morphism of triangulated functors.
\end{convention}

\section{Functors} \label{s:02}

\begin{definition}(\cite[Definition 2.2]{VY}) \label{d:01}
Let $\cat{C}$ be a category. 
\begin{enumerate}
\item A \emph{copointed functor} on $\cat{C}$ is a pair $(F,\si)$, where $F:\cat{C}\to \cat{C}$ is a functor and $\si: F\to \opn{id}_{\cat{C}}$ is a morphism of functors.
\item A copointed functor $(F,\si)$ is called \emph{idempotent} if the morphisms $F(\si_{C}): F(F(C))\to F(C)$ and  $\si_{F(C)}: F(F(C))\to F(C)$ are isomorphisms for all objects $C \in \cat{C}$.
\item If $\cat{C}$ is a triangulated category, $F$ a triangulated functor and $\si$ a morphism of triangulated functors, we call $(F,\si)$ a \emph{copointed triangulated functor.}
\end{enumerate}
\end{definition}

\begin{definition}(\cite[Definition 2.8]{VY}) \label{d:07}
Let $\cat{C}$ be a category.
\begin{enumerate}
\item A \emph{pointed functor} on $\cat{C}$ is a pair $(G, \tau)$, where $G:\cat{C}\to \cat{C}$ is a functor and $\tau: \opn{id}_{\cat{C}}\to G$ is a morphism of functors.
\item A pointed functor $(G,\tau)$ is called \emph{idempotent} if the morphisms $G(\tau_{C}): G(C)\to G(G(C))$ and $\tau_{G(C)}: G(C)\to G(G(C))$ are isomorphisms for all objects $C\in \cat{C}$.
\item If $\cat{C}$ is a triangulated category, $G$ a triangulated functor and $\tau$ a morphism of triangulated functors, we call $(G,\tau)$ a \emph{pointed triangulated functor.}
\end{enumerate}
\end{definition}

\begin{remark}
Some of notions formalised in Definition \ref{d:01} appear in the literature under other names. Kashiwara and Schapira call idempotent pointed functors \emph{projectors} (\cite[Definition 4.1]{KS}).    Other terminology comes from the theory of monads: the reader may convince themselves that the data of an idempotent copointed functor is precisely that of an \emph{idempotent comonad}. Another name for the same notion is a \emph{colocalization functor} e.g. \cite[2.4]{Kr}. Unlike \cite{Kr}, we have included the morphism $\si$ in the data of our definition, as opposed to simply mandating its existence. 
\end{remark}

\begin{lemma} \label{l:01}
Let $(F,\si)$ be an idempotent copointed functor on a category $\cat{C}$. Then, the morphisms $F(\si_{C})$ and $\si_{F(C)}$ are equal for all objects $C\in \cat{C}$.
\end{lemma}
\begin{proof}
This is \cite[Proposition 4.1.2]{KS}, dualised to the copointed case.
\end{proof}

 \begin{definition} \label{d:10} Let $\cat{C}$ be a category.
 \begin{enumerate}
 \item If $(F,\si)$ is an idempotent copointed functor on $\cat{C}$, define \[\cat{C}_{F}:= \{C\in \cat{C} \ \vert \ \si_{C} \ \text{is an isomorphism}\}.\]
 \item If $(G,\tau)$ is an idempotent pointed functor on $\cat{C}$, define \[\cat{C}_{G}:= \{C\in \cat{C} \ \vert \ \tau_{C} \ \text{is an isomorphism}\}.\]
 \end{enumerate}
 \end{definition}
 
 \begin{remark}
 In the literature, the category $\cat{C}_{F}$ is often called the category of $F$-colocal objects, and  $\cat{C}_{G}$ the category of $G$-local  objects (\cite[2.5]{Kr}). 
 \end{remark}
 
 The following lemma is standard, but we provide a proof for completeness.
 
 \begin{lemma} \label{l:08}
 Let $(F,\si)$ be an idempotent copointed functor on a category $\cat{C}$. For any $C\in \cat{C}$, $F(C)\in \cat{C}_{F}$. The functor $F:\cat{C}\to \cat{C}_{F}$ is right adjoint to the inclusion $\opn{inc}:\cat{C}_{F}\to \cat{C}.$
 \end{lemma}
 \begin{proof}
 From the definition of an idempotent copointed functor, it is clear that $F(C)\in \cat{C}_{F}$ for any $C\in \cat{C}$.
 
 Suppose $C\in \cat{C}_{F}$, and $C'\in \cat{C}$. Consider the following two maps: \[\mathrm{Hom}_{\cat{C}}(C,F(C'))\to \mathrm{Hom}_{\cat{C}}(C,C'),\] \[f\mapsto \si_{C'}\circ f,\] and \[\mathrm{Hom}_{\cat{C}}(C,C')\to \mathrm{Hom}_{\cat{C}}(C,F(C')),\] \[f\mapsto F(f)\circ \si_{C}^{-1}.\] Verifying that these two correspondences are natural in both variables and inverse to each other is straightforward.
 \end{proof}
 
 There is, of course, a dual version to Lemma \ref{l:08}. We state it here for completeness, but omit the proof.
 
 \begin{lemma} \label{l:09}
 Let $(G,\tau)$ be an idempotent pointed functor on a category $\cat{C}$. For any $C\in \cat{C}$, $G(C)\in \cat{C}_{G}$. The functor $G:\cat{C}\to \cat{C}_{G}$ is left adjoint to the inclusion $\opn{inc}:\cat{C}_{G}\to \cat{C}.$ \qed
\end{lemma}

 Let $A$ and $B$ be rings, and $F:\cat{M}(A)\to \cat{M}(B)$ a left exact functor.  $F$ admits a right derived functor $(\mrm{R}F, \eta)$, where $\mrm{R}F:\cat{D}(A)\to \cat{D}(B)$ and $\eta: F\to \mrm{R}F$. Recall that a complex $M\in \cat{K}(A)$ is said to be \emph{right} $F$\emph{-acyclic} if the map $\eta_{M}:F(M)\to \mrm{R}F(M)$ is an isomorphism in $\cat{D}(B)$. If $M\in \cat{M}(A)$, $M$ is right $F$-acyclic if and only if $\mrm{R}^{i}F(M)=0$ for all $i\geq 1$, where $\mrm{R}^{i}F:= \opn{H}^{i}(\mrm{R}F)$.  The \emph{right cohomological dimension} of $F$ is $\opn{sup}\{i\in \mathbb{N} \vert \ \mrm{R}^{i}F \neq 0\}.$
 
 \section{Torsion classes and weak stability} \label{s:07}

Let $A$ be a ring. A \emph{torsion class} in $\cat{M}(A)$ is a class of objects $\cat{T}\subseteq \cat{M}(A)$ closed under submodules, extensions, quotients, and arbitrary direct sums. Every torsion class $\cat{T}$ is associated with a left exact additive functor $\Gamma_{\cat{T}}:\cat{M}(A)\to \cat{M}(A)$: if $M\in \cat{M}(A)$, $\Gamma_{\cat{T}}(M)$ is the largest submodule of $M$ which lies in the $\cat{T}$. The canonical inclusions $\Gamma_{\cat{T}}(M)\to M$, as $M$ varies over $\cat{M}(A)$, assemble together to define a morphism of functors $\si: \Gamma_{\cat{T}}\to \opn{id}_{\cat{M}(A)}$. $(\Gamma_{\cat{T}},\si)$ is an idempotent copointed functor.

\begin{remark}
The notion we have described above is often called a \emph{hereditary} torsion class in the literature, with the term `hereditary' signifying closure under subobjects. Here, we have dropped the adjective.
\end{remark}

Every torsion class $\cat{T}$ can be associated with a set of left ideals of $A$, $\opn{Filt}(\cat{T})$, called the \emph{Gabriel filter} of $\cat{T}$. A left ideal $\mathfrak{i}\in \opn{Filt}(\cat{T})$ if and only if $A/\mathfrak{i} \in \cat{T}$. The functor $\Gamma_{\cat{T}}$ also admits a description in terms of $\opn{Filt}(\cat{T})$: if $M\in \cat{M}(A)$,
\begin{equation} \label{e:01}
\Gamma_{\cat{T}}(M) = \varinjlim_{\mathfrak{i}\in \opn{Filt}(A)} \mrm{Hom}_{A}(A/\mfrak{i},M).
\end{equation}
Note that $\opn{Filt}(\cat{T})$ is a directed poset under reverse inclusion. 

By the defining property of right derived functors, the morphism of functors $\si: \Ga_{\cat{T}}\to \opn{id}_{\cat{K}(A)}$ induces a morphism of triangulated functors 

\begin{equation} \label{e:15}
\si^{\mrm{R}}: \mrm{R}\Ga_{\cat{T}} \to \opn{id}_{\cat{D}(A)}.
\end{equation}

The pair $(\mrm{R}\Ga_{\cat{T}}, \si^{\mrm{R}})$ is a copointed functor on $\cat{D}(A)$.

Torsion classes are ubiquitous.  For example, if $\mfrak{a}$ is an ideal in $A$, there is the class 
\begin{equation} \label{e:05}
\cat{T}_{\mfrak{a}}: = \{M\in \cat{M}(A)\ \vert \  \forall m\in M, \ \exists n\in \mathbb{N} \ \text{such that} \  \mfrak{a}^{n}m = 0\}.
\end{equation}

When $\mfrak{a}$ is finitely generated as a left ideal, $\cat{T}_{\mfrak{a}}$ is indeed a torsion class (\cite[Definition 3.3]{VY}) with associated torsion functor $\Ga_{\mfrak{a}} : = \Ga_{\cat{T}_{\mfrak{a}}}$, where $\Ga_{\mfrak{a}}$ picks out all the elements of a module which are annihilated by some power of $\mfrak{a}$. For another example, consider a left Ore set $S\subseteq A$. We can associate a torsion class $\cat{T}_{S}$ to $S$: 
 \begin{equation} \label{e:06}
 \cat{T}_{S} = \{M\in \cat{M}(A)\ \vert\  \forall m\in M,\ \exists s\in S \ \text{such that} \ sm= 0\}.
 \end{equation}
 
  $\Ga_{S} := \Ga_{\cat{T}_{S}}$ selects the elements of a module which are annihilated by some element of $S$. We will study $\cat{T}_{S}$ further in \S \ref{s:03}.
  
\begin{definition} (\cite[Definition 3.4]{VY}) \label{d:05}
Let $\cat{T}$ be a torsion class in $\cat{M}(A)$. $\cat{T}$ is said to be \emph{finite dimensional} if the functor $\Gamma_{\cat{T}}$ has finite right cohomological dimension.
\end{definition}

\begin{definition} (Gabriel, \cite[Chapter IV.7]{Ste}) \label{d:03}
Let $\cat{T}$ be a torsion class in $\cat{M}(A)$. $\cat{T}$ is called a \textit{stable} torsion class if for any injective $A$-module $E$, $\Gamma_{\cat{T}}(E)$ is injective. 
\end{definition}

A torsion class $\cat{T}$ is stable if and only if $\cat{T}$ is closed under essential extensions (\cite[Proposition 7.1]{Ste}). Goodearl and Jordan studied stability (without explicitly using the terminology) in the context of left denominator sets: in this setting, they further characterised it as the preservation of essentiality under localisation (\cite[Theorem 1]{GJ}). Perhaps the most important manifestation of stability is when it occurs in the torsion class associated to an ideal: if $\mfrak{a}$ is a finitely generated ideal with associated torsion class $\cat{T}_{\mfrak{a}}$, $\cat{T}_{\mfrak{a}}$ is stable precisely when $\mfrak{a}$ satisfies the \emph{left Artin-Rees} property (\cite[Lemma 13.1]{GW}).

\begin{definition} \cite[Definition 3.4]{VY} \label{d:02}
Let $\cat{T}$ be a torsion class in $\cat{M}(A)$. $\cat{T}$ is called a \textit{weakly stable} torsion class if for any injective $A$-module $E$, $\Gamma_{\cat{T}}(E)$ is right $\Gamma_{\cat{T}}$-acyclic. 
\end{definition}

Obviously, every stable torsion class is weakly stable. The converse is not true (\cite[Example 3.11]{VY}).

Weakly proregular sequences were introduced in commutative algebra by Alonso, Jeremias, and Lipman (\cite{AJL}). The notion of a weakly stable torsion class was introduced by Yekutieli and the author in \cite{VY}, with the intention of providing a characterisation of weak proregularity that could be generalised to the context of noncommutative ring theory (\cite[Theorem 4.12]{VY}). This, in turn, led to a noncommutative generalisation of the \emph{MGM equivalence} (\cite[Theorem 8.3]{VY}).

\begin{proposition} \label{p:02} \label{p:04}
Let $\cat{T}$ be a torsion class in $\cat{M}(A)$. Then, the following are equivalent:
\begin{enumerate}
\rmitem{i} $\cat{T}$ is weakly stable.
\rmitem{ii} $(\mrm{R}\Ga_{\cat{T}}, \si^{\mrm{R}})$ is an idempotent copointed functor on $\cat{D}^{+}(A)$.
\rmitem{iii} The morphism $\mrm{R}\Ga_{\cat{T}}(\si^{\mrm{R}}_{M}):\mrm{R}\Ga_{\cat{T}}(\mrm{R}\Ga_{\cat{T}}(M))\to \mrm{R}\Ga_{\cat{T}}(M)$ is an isomorphism for any $M\in \cat{D}^{+}(A)$.
\rmitem{iv} The morphism $\si^{\mrm{R}}_{\mrm{R}\Ga_{\cat{T}}(M)}:\mrm{R}\Ga_{\cat{T}}(\mrm{R}\Ga_{\cat{T}}(M))\to \mrm{R}\Ga_{\cat{T}}(M)$ is an isomorphism for any $M\in \cat{D}^{+}(A)$.
\end{enumerate}
If $\cat{T}$ is finite dimensional, all instances of $\cat{D}^{+}(A)$ in the above statement can be replaced by $\cat{D}(A)$.
\end{proposition}
\begin{proof}
(i) $\implies$ (ii): The proof in \cite[Theorem 2.6]{VY} works in this context as well.

(i) $\implies$ (iii), (iv): Both of these implications follow from the definition of weak stability.

(iii) $\implies$ (i): Let $I$ be an injective $A$-module. By hypothesis, the map 
\begin{equation} \label{e:04}
\mrm{R}\Ga_{\cat{T}}(\si^{\mrm{R}}_{I}):\mrm{R}\Ga_{\cat{T}}(\mrm{R}\Ga_{\cat{T}}(I))\to \mrm{R}\Ga_{\cat{T}}(I)
\end{equation}
 is an isomorphism. Since $I$ is injective, $\Ga_{\cat{T}}(I)\cong \mrm{R}\Ga_{\cat{T}}(I)$; thus, \ref{e:04} implies that $\opn{H}^{i}(\mrm{R}\Ga_{\cat{T}}(\Ga_{\cat{T}}(I))) = 0$ for all $i\geq 1$. This precisely says that $\Ga_{\cat{T}}(I)$ is right $\Ga_{\cat{T}}$ acyclic.
 
 (iv) $\implies$ (i): The proof of this is very similar to the proof of the implication (iii) $\implies$ (i).
 
If $\cat{T}$ is finite dimensional, the implication (i) $\implies$ (ii) is precisely the content of \cite[Theorem 2.6]{VY}. The rest of the proof is exactly as asbove.
\end{proof}

\begin{definition} (\cite[Definition 3.4]{VY})\label{d:08}
Let $\cat{T}$ be a torsion class in $\cat{M}(A)$). $\cat{T}$ is called \emph{quasi-compact} if each of the functors $\mrm{R}^{i}\Ga_{\cat{T}}: \cat{M}(A)\to \cat{M}(A)$, for $i\in \mathbb{N}$, commutes with arbitary direct sums.
\end{definition}

The reader may check that $\cat{T}$ is quasi-compact precisely when the functor $\mrm{R}\Ga_{\cat{T}}$ commutes with all direct sums which exist in $\cat{D}^{+}(A)$. If $\cat{T}$ is finite dimensional, then it is quasi-compact if and only if the functor $\mrm{R}\Ga_{\cat{T}}$ commutes with all direct sums in $\cat{D}(A)$. See \cite[Theorem 1.11]{VY}.

\begin{lemma} \label{l:02}
 Let $\cat{T}$ be a torsion class in $\cat{M}(A)$. Suppose the cohomological dimension of $\Gamma_{\cat{T}}$ is less than or equal to one. 
Then, $\Gamma_{\cat{T}}$ is weakly stable.
\begin{proof}
Let $E$ be an injective $A$-module. Consider the exact sequence $$0\rightarrow \Gamma_{\cat{T}}(E)\rightarrow E\rightarrow E/\Gamma_{\cat{T}}(E)\rightarrow 0.$$ 

The long exact sequence in cohomology arising from applying the derived functor $\mrm{R}\Ga_{\cat{T}}$ to the short exact sequence above begins with the following terms: \[0\to \Ga_{\cat{T}}(\Ga_{\cat{T}}(E))\to \Ga_{\cat{T}}(E) \to \Ga_{\cat{T}}(E/\Ga_{\cat{T}}(E)) \to \opn{H}^{1}(\mrm{R}\Gamma_{\cat{T}}(\Gamma_{\cat{T}}(E)))\to \opn{H}^{1}(\mrm{R}\Ga_{\cat{T}}(E)). \]

It follows that $\opn{H}^{1}(\mrm{R}\Gamma_{\cat{T}}(\Gamma_{\cat{T}}(E)))=0$. Our assumption on the cohomological dimension of $\Gamma_{\cat{T}}$ now tells us that $\Gamma_{\cat{T}}(E)$ is right $\Gamma_{\cat{T}}$-acyclic.
\end{proof}
\end{lemma}

\section{Examples} \label{s:03}

\subsection*{Left denenominator sets}
Every left denominator set $S$ in a ring $A$ gives rise to an associated torsion theory 
$\cat{T}_{S}$, with corresponding torsion functor $$\Ga_{S}(M) := 
\Gamma_{\cat{T}_{S}}(M)=\{m\in M\ |\ sm=0\  \text{for some}\  s\in S\}.$$
As in \ref{e:15}, there is the morphism of functors $\si^{\mrm{R}}: \mrm{R}\Gamma_{S} \to \opn{id}_{\cat{D}(A)}$.

\begin{definition} \label{d:04}
Let $\opn{K}(A;S):= ( \cdots \to 0 \to A\to A_{S} \to 0 \to \cdots)$ denote the two term complex in $\cat{C}(A^{\mrm{en}})$ with $A$ in degree $0$ and $A_{S}$ in degree $1$. The differential in $\opn{K}(A;S)$ is the canonical localisation map.
\end{definition}

There is a morphism of complexes $e : \opn{K}(A;S)\to A$, sending $1\in \opn{K}(A;S)^{0}$ to $1\in A$,  in $\cat{C }(A^{\opn{en}})$. In $\cat{K}(A^{\mrm{en}})$, $e$ fits into a distinguished triangle \[ \opn{K}(A;S)\xrightarrow{e} A \to A_{S} \to \opn{K}(A;S)[1],\]
 where the morphism $A\to A_{S}$ is again the canonical localisation map.

\begin{lemma} \label{l:03}
Let $S$ be a left denominator set in a ring $A$. Then, there is a unique morphism \[\nu: \mrm{R}\Gamma_{S} \to \opn{K}(A;S)\ot_{A}^{\mrm{L}}(-)\] of triangulated functors from $\cat{D}(A)$ to itself, such that for any $M\in \cat{D}(A)$, $(e \ot_{A}^{\mrm{L}} \opn{id}_{M})\circ \nu_{M} =\si_{M}^{\mrm{R}}.$
\end{lemma}
\begin{proof}
Let $M, N\in \cat{D}(A)$. There is an exact triangle \[\mrm{K}(A;S)\ot_{A}^{\mrm{L}}M\to M\to M_{S}\to (\mrm{K}(A;S)\ot_{A}^{\mrm{L}}M)[1]\] in $\cat{D}(A).$ The first map in this triangle is $e \ot_{A}^{\mrm{L}} \opn{id}_{M}.$

 We claim that $\mrm{Hom}_{\cat{D}(A)}(\mrm{R}\Gamma_{S}(N),M_{S}[i])=0$ for all $i\in \mathbb{Z}$. This is easy:  \[\mrm{Hom}_{\cat{D}(A)}(\mrm{R}\Gamma_{S}(N),M_{S}[i]) = \mrm{Hom}_{\cat{D}(A_{S})}(A_{S}\ot_{A}^{\mrm{L}}\mrm{R}\Gamma_{S}(N),M_{S}[i]),\] and $A_{S}\ot_{A}^{\mrm{L}}\mrm{R}\Gamma_{S}(N)=0$.

Recall that $\mrm{Hom}_{\cat{D}(A)}(\mrm{R}\Ga_{s}(N),-)$ is a cohomological functor from $\cat{D}(A)$ to $\cat{M}(\K)$. By considering the resulting long exact sequence in cohomology,  we note that $e\ot_{A}^{\mrm{L}}\opn{id}_{M}$ induces an isomorphism \[\mrm{Hom}_{\cat{D}(A)}(\mrm{R}\Gamma_{S}(N),\mrm{K}(A;S)\ot_{A}^{\mrm{L}}M)\xrightarrow{\cong} \mrm{Hom}_{\cat{D}(A)}(R\Gamma_{S}(N),M).\] 

Taking $N=M$, we see that there is a unique map \[\nu_{M}: \mrm{R}\Gamma_{S}(M)\to \mrm{K}(A;S)\ot_{A}^{\mrm{L}}M\] such that  ${(e\ot_{A}^{\mrm{L}} \opn{id}_{M})}\circ {{\nu}}_{M}= {\si^{\mrm{R}}_{M}}.$ 

Allowing $N$ to vary arbitrarily over $\cat{D}(A)$, we conclude that $\nu$ is a morphism of functors: if $f: N\to M$ is a morphism in $\cat{D}(A)$, then both $(\opn{id}_{\mrm{K}(A;S)}\ot_{A}^{\mrm{L}} f)\circ {\nu}_{N}$ and ${{\nu}_{M}} \circ \mrm{R}\Ga_{S}(f)$ are sent to ${\si^{\mrm{R}}_{M}} \circ \mrm{R}\Ga_{S}(f) = f \circ {\si^{\mrm{R}}_{N}}$ after composing with $e\ot_{A}^{\mrm{L}} \opn{id}_{M}.$ A similar argument shows that $\nu$ is a morphism of triangulated functors.
\end{proof}

\begin{remark} \label{r:02}
The style of argument used in the proof of Lemma \ref{l:03} is quite standard in the theory of localisation for triangulated categories (\cite[Proposition 4.9.1, Proposition 4.11.1]{Kr}).
\end{remark}

\begin{theorem} \label{p:03} \label{t:02}
Let $S$ be a left denominator set in a ring $A$. Then, the following are equivalent:
\begin{enumerate}
\rmitem{i} The torsion class $\cat{T}_{S}$ is weakly stable.
\rmitem{ii} For any right $\Ga_{S}$-acyclic module $M$, the canonical map $M\to M_{S}$ is surjective.
\rmitem{iii} For any injective module $E$, the canonical map $E\to E_{S}$ is surjective.
\rmitem{iv} $\nu:  \mrm{R}\Gamma_{S} \to \opn{K}(A;S)\ot_{A}^{\mrm{L}}(-)$ is an isomorphism of functors.
\rmitem{v} $\Gamma_{S}$ has cohomological dimension less than or equal to one.
\end{enumerate}
\end{theorem}
\begin{proof}
(i) $\implies$ (ii): Suppose $\cat{T}_{S}$ is weakly stable, and let $M$ be a right $\Ga_{S}$-acyclic module. We wish to show that the natural map $M\to M_{S}$ is surjective.

Suppose it is not, and denote its (non-zero) cokernel by $J$; note that $\Ga_{S}(J)=J$. If $J^{3}$ is an injective hull of $J$, we can use the inclusion $J\to J^{3}$ to construct an exact sequence 
\begin{equation} \label{e:02}
0\to \Ga_{S}(M)\to M\to M_{S} \to J^{3}.
\end{equation}

As a right $A$-module, $A_{S}$ is flat. Thus, if $M_{S}\to K$ is an injective resolution of $M_{S}$ as $A_{S}$-modules, $K$ restricts to a complex of injective $A$-modules and can be used to compute $\mrm{R}\Ga_{S}(M_{S}).$ As every term of $K$ is an $A_{S}$-module, the complex $\Ga_{S}(K)=0$. Therefore $M_{S}$ is right $\Ga_{S}$-acyclic.
 
It follows that the three term complex $M\to M_{S} \to J^{3}$ excised from \ref{e:02} can be extended to a right $\Ga_{S}$-acyclic resolution of $\Ga_{S}(M):$ 
\begin{equation} \label{e:03}
M\to M_{S} \to J^{3}\to J^{4} \to \ldots.
\end{equation}

We can use the complex \ref{e:03} to compute $\mrm{R}\Ga_{S}(\Ga_{S}(M)).$ $0\neq J = \ker(\Ga_{S}(J^{3})\to \Ga_{S}(J^{4})).$ 
Since $\Ga_{S}(M_{S}) = 0$, we conclude that $\mrm{R}^{3}\Ga_{S}(\Ga_{S}(M))\neq 0.$ This contradicts the fact that $\Ga_{S}(M)$ is right $\Ga_{S}$-acyclic.

(ii) $\implies$ (iii): This is obvious. 

(iii) $\implies$ (iv): Assume that the natural map $I\to I_{S}$ is surjective, for any injective module $I$. We wish to show that $\nu$ is a natural isomorphism of functors.

 It is enough to show that \[{\nu}_{I}:\Ga_{S}(I)\to \opn{K}(A;S)\ot_{A}^{\mrm{L}} I\] is a quasi-isomorphism for every K-injective complex of injective modules $I$. The uniqueness of ${\nu}$ (Lemma \ref{l:03}) implies that ${\nu}_{{I}}$ is  the image in $\cat{D}(A)$ of the map $\xi_{I}: \Ga_{S}(I)\to \opn{K}(A;S)\ot_{A} I$ sending  $x\in \Ga_{S}(I^{i})$ to $(x,0)\in I^{i}\oplus I_{S}^{i-1} = (\opn{K}(A;S)\ot_{A} I)^{i}.$
 
 We wish to show that $\xi_{I}$ is a quasi-isomorphism. Since $\opn{K}(A;S)$ is a bounded complex of flat $A$-modules, it is enough to show this for a single injective module $E$ (\cite[Chapter 1.7]{RD} or \cite[Proposition 1.9]{PSY1}). However, this is precisely our assumption.

(iv) $\implies$ (v): This is obvious. 

(v)  $\implies$  (i): 
This follows from Lemma \ref{l:02}.
\end{proof}

\subsection*{Strongly idempotent ideals}
Let $\mfrak{a}$ be an ideal in $A$. In general, $\cat{T}_{\mfrak{a}}$ will not be a torsion class unless $\mfrak{a}$ is finitely generated as a left ideal. However, there is a specific instance where the finite generatedness of $\mfrak{a}$ is irrelevent: when $\mfrak{a}$ is idempotent i.e.~$\mfrak{a}  = \mfrak{a}^{2}.$ When $\mfrak{a}$ is idempotent, $\cat{T}_{\mfrak{a}}$ reduces to the class of modules $\{M\in \cat{M}(A)\ \vert \ \mfrak{a}M = 0\}$, and the functor $\Ga_{\mfrak{a}}$ is nothing but $\mrm{Hom}_{A}(A/\mfrak{a}, -)$. Idempotent ideals provide an effective `toy model' to test possible conjectures.

For any ideal $\mfrak{a}$, there is a canonical morphism \[\mu_{\a}: A/\mfrak{a} \otimes_{A}^{\mrm{L}} A/\mfrak{a} \to A/\mfrak{a}\] in $\cat{D}(A),$ induced from the multiplication $A/\mfrak{a} \otimes_{A} A/\mfrak{a}\to A/\mfrak{a}$. The reader may note that $\mfrak{a}$ is idempotent precisely when $\mrm{Tor}_{1}^{A}(A/\mfrak{a},A/\mfrak{a}) = \opn{H}^{-1}(A/\mfrak{a}\ot_{A}^{\mrm{L}} A/\mfrak{a}) =0.$ This motivates the following definition.

\begin{definition} [\cite{APT1}] \label{d:06}
Let $\mfrak{a}$ be an ideal in $A$. $\mfrak{a}$ is said to be \emph{strongly idempotent} when the canonical morphism $\mu_{A}:A/\mfrak{a}\otimes_{A}^{\mrm{L}} A/\mfrak{a} \to A/\mfrak{a}$ is an isomorphism in $\cat{D}(A)$.
\end{definition}

\begin{remark} \label{r:01}
Where should the morphism $A/\mfrak{a}\otimes_{A} ^{\mrm{L}} A/\mfrak{a} \to A/\mfrak{a}$ live? Through a judicious choice of resolution, we can construct it as a map in either $\cat{D}(A)$, $\cat{D}(A^{\mrm{op}})$, or $\cat{D}(\K)$ - what matters is that the choice of ambient category does not affect whether the map is an isomorphism or not. When $A$ is flat over $\K$, there is an easy and definitive answer to this question: the category $\cat{D}(A^{\mrm{en}})$. When $A$ is not flat over $\K$, there is also a definitive answer: the category $\cat{D}(A\ot^{\mrm{L}} A^{\mrm{op}})$. However, constructing  $\cat{D}(A\ot^{\mrm{L}} A^{\mrm{op}})$ is a delicate task: one first finds a dg quasi-ismorphism $\tilde{A}\to A$ where $\tilde{A}$ is non-positive and termwise flat over $\K$, and then defines $\cat{D}(A\ot^{\mrm{L}} A^{\mrm{op}}): = \cat{D}(\tilde{A}\ot \tilde{A}^{\mrm{op}})$. It turns out that this construction is independent  of choice of resolution. See \cite{AILN}, \cite{Ye3}, \cite{VY2}.
\end{remark}

 \begin{lemma} \label{l:13}
Let $\mfrak{a}$ be an ideal in $A$. Then, $\mfrak{a}$ is strongly idempotent if and only if $\opn{H}^{i}(A/\mfrak{a}\otimes_{A}^{\mrm{L}} A/\mfrak{a}) = 0$ for all $i\leq -1$.
\end{lemma}
\begin{proof}
$\opn{H}^{0}(\mu_{\a})$ is always an isomorphism, and $\opn{H}^{i}(A/\a)=0$ for $i\leq -1$. 
\end{proof}

\begin{proposition} \label{p:05}
Let $\mfrak{a}$ be an ideal in $A$. Then, $\mfrak{a}$ is strongly idempotent if and only if the torsion class $\cat{T}_{\mfrak{a}}$ is weakly stable.
\end{proposition}
\begin{proof}
By Lemma \ref{l:13}, $\a$ is strongly idempotent if and only if $\mrm{Hom}_{\cat{D}(A)}(A/\mfrak{a}\otimes_{A}^{\mrm{L}} A/\mfrak{a}, E[n]) = 0$ for all injective $A$-modules $E$, and $n\neq 0$. By the adjuction between $\mrm{RHom}_{A}$ and $\otimes_{A}^{\mrm{L}}$, \[\mrm{Hom}_{\cat{D}(A)}(A/\mfrak{a}\otimes_{A}^{\mrm{L}} A/\mfrak{a}, E[n]) \cong \mrm{Hom}_{\cat{D}(A)}(A/\mfrak{a}, \mrm{RHom}_{A}(A/\mfrak{a},E[n])).\]

Moreover, when $E$ is injective\[\mrm{Hom}_{\cat{D}(A)}(A/\mfrak{a}, \mrm{RHom}_{A}(A/\mfrak{a},E[n])) \cong \mrm{Hom}_{\cat{D}(A)}(A/\mfrak{a}, \mrm{Hom}_{A}(A/\mfrak{a},E)[n]),\] and \[\mrm{Hom}_{\cat{D}(A)}(A/\mfrak{a}, \mrm{Hom}_{A}(A/\mfrak{a},E)[n]) \cong \opn{H}^{n}(\mrm{RHom}_{A}(A/\mfrak{a}, \mrm{Hom}_{A}(A/\mfrak{a},E))).\]

Thus $\mu_{\a}: A/\mfrak{a} \otimes_{A}^{\mrm{L}} A/\mfrak{a} \to A/\mfrak{a}$ is an isomorphism in $\cat{D}(A)$ if and only if, for all injective $E\in \cat{M}(A)$, $\mrm{Hom}_{A}(A/\mfrak{a},E)$ is right acyclic with respect to the functor $\mrm{Hom}_{A}(A/\mfrak{a},-)$; this is precisely what it means for $\cat{T}_{\mfrak{a}}$ to be weakly stable.
\end{proof}

\begin{example} 
In \cite[Example 3]{APT1}, Auslander, Platzeck, and Todorov construct a strongly idempotent ideal 
$\a$ in a left and right Artinian ring $A$ such that $A/\a$ has infinite left projective dimension. Therefore $\Gamma_{\a} = \mrm{Hom}_{A}(A/\a, -)$ has infinite cohomological dimension.  

Since $A/\a$ has infinite projective dimension, it is not perfect. Thus $\mrm{R}\Ga_{\mfrak{a}}$ does not preserve direct sums in $\cat{D}(A)$.  However, $\cat{T}_{\mfrak{a}}$ \emph{is} quasi-compact, as $\mrm{R}\Ga_{\a}$ does preserve coproducts which exist in  $\cat{D}^{+}(A).$ In the next example, we show that even this cannot be taken for granted. 
\end{example}

\begin{example}
 Let $A$ be a non-noetherian Von Neumann regular commutative ring of finite global dimension; an example of such a ring is  $ \prod_{i\in \mathbb{N}} \mathbb{F}_{2}.$ As $A$ is Von Neumann regular, every $A$-module, left or right, is flat over $A$. Therefore, \emph{every} ideal in $A$ is strongly idempotent.
 
Since $A$ is not noetherian, there is a family of injective modules $\{E_{i}\}_{i\in \mathbb{N}}$ such that $\bigoplus_{i\in \mathbb{N}}E_{i}$ is not injective. By Baer's criterion, there is an ideal $\a$ in $A$ such that \[\mrm{Ext}^{1}_{A}(A/\a, \bigoplus_{i\in \mathbb{N}}E_{i})\neq 0.\]

It follows that $\mrm{R}\Gamma_{\a}$ does not preserve coproducts, even in $\cat{D}^{+}(A).$ Moreover, in this example $\Gamma_{\a}$ has finite cohomological dimension, since $A$ has finite global dimension.
\end{example}

\begin{example}
Every ideal in a commutative noetherian ring generates a weakly stable torsion class. This is not true in the noncommutative context. As an example, consider the augmentation ideal of $\mathcal{U}(\opn{sl}_{2}(\mathbb{C}))$, which we denote by $\mathcal{I}$. It is well known that $\mathcal{I}$ is idempotent but not strongly idempotent. In particular, $\mrm{Tor}_{3}^{\mathcal{U}(\opn{sl}_{2}(\mathbb{C}))}(\mathcal{U}(\opn{sl}_{2}(\mathbb{C}))/\mathcal{I}, \mathcal{U}(\opn{sl}_{2}(\mathbb{C}))/\mathcal{I})\neq 0.$
\end{example}

\section{Localisation} \label{s:04}

Let $\cat{T}$ be a torsion class in $\cat{M}(A)$. In \S \ref{s:02} we recalled how $\cat{T}$ can be associated with a functor $\Gamma_{\cat{T}}:\cat{M}(A) \to \cat{M}(A)$ and a set of left ideals $\opn{Filt}(\cat{T})$. In this section, we consider another functor that can be associated with $\cat{T}$: localisation. We will begin with a short overview of the definitions involved, recalling only the notions directly relevant to our work.  For full details see \cite[Chapter IX]{Ste}.

If $\i \in \opn{Filt}(\cat{T})$ and $a\in A$, the left ideal $a^{-1}\cdot \i := \{x\in A \ \vert \ xa\in \i\}\in \opn{Filt}(\cat{T}).$ From this, it follows that for any $M\in \cat{M}(A)$, the directed limit \[\varinjlim_{\i \in \opn{Filt}(\cat{T})} \mathrm{Hom}_{A}(\i,M)\] carries the natural structure of a left $A$-module: if $a\in A$ and $f: \i\to M$ is a representative of an element in $\varinjlim_{\i \in \opn{Filt}(\cat{T})} \mathrm{Hom}_{A}(\i,M)$, then $a\cdot f$ is given by the equivalence class of \[a^{-1}\cdot \i \xrightarrow{\cdot a}  \i \xrightarrow{f} M.\]

\begin{definition} \label{d:09}
Let $\cat{T}$ be a torsion class in $\cat{M}(A)$. Define the functor \[\Ja_{\cat{T}}:= \cat{M}(A)\to \cat{M}(A)\] by \[\Ja_{\cat{T}}(M) :=  \varinjlim_{\i \in \opn{Filt}(\cat{T})} \mathrm{Hom}_{A}(\i,M)\] for $M\in \cat{M}(A)$.
\end{definition}

If $M\in \cat{M}(A)$, there is a morphism $\tau_M: M\to \Ja_{\cat{T}}(M)$, arising from the canonical identification of $M$ with $\mathrm{Hom}_{A}(A,M)$. These morphisms assemble together to define a natural transformation $\tau: \opn{id}_{\cat{M}(A)}\to \Ja_{\cat{T}}$. The pair $(\Ja_{\cat{T}}, \tau)$ is a pointed functor. It is \emph{not} always the case that $(\Ja_{\cat{T}}, \tau)$ is idempotent. However, the following is true:

\begin{lemma} \label{l:07}
Let $\cat{T}$ be a torsion class in $\cat{M}(A)$. The morphisms $\tau_{\Ja_{\cat{T}}(M)}: \Ja_{\cat{T}}(M)\to \Ja_{\cat{T}}(\Ja_{\cat{T}}(M))$ and $\Ja_{\cat{T}}(\tau_{M}): \Ja_{\cat{T}}(M)\to \Ja_{\cat{T}}(\Ja_{\cat{T}}(M))$ coincide for any object $M\in \cat{M}(A)$.
\end{lemma}
\begin{proof}
This is the content of \cite[Lemma IX.1.4]{Ste}.
\end{proof}

The morphism of functors $\tau: \opn{id}_{\cat{M}(A)}\to \Ja_{\cat{T}}$ induces a morphism of functors $\tau^{R} : \opn{id}_{\cat{D}(A)}\to \mrm{R}\Ja_{\cat{T}}$. If $M\in \cat{D}(A)$ and $M\to I$ is a K-injective resolution of $M$, $\tau_{M}^{\mrm{R}}$ is the composite $M\to I\xrightarrow{\tau_{I}} \Ja_{\cat{T}}(I)$.

The functor $\Ja$ was introduced by Gabriel as a part of his theory of localisation for torsion classes. Suppose $A$ is a ring, and $\cat{T}$ a torsion class in $\cat{M}(A)$ with associated Gabriel filter $\opn{Filt}(\cat{T})$.  A module $M\in \cat{M}(A)$ is torsion-free, with respect to $\cat{T}$, precisely when the canonical map $\mrm{Hom}_{A}(A,M)\to \mrm{Hom}_{A}(A/\i,M)$ is injective for every left ideal $\i\in \opn{Filt}(A)$. Define $M$ to be $\cat{T}$\emph{-divisible} if and only if the canonical map $\mrm{Hom}_{A}(A,M)\to \mrm{Hom}_{A}(A/\i,M)$ is surjective for every left ideal $\i\in \opn{Filt}(A)$, and $\cat{T}$\emph{-closed} if it is both $\cat{T}$-divisible and torsion-free with respect to $\cat{T}$. 

 The localisation of $\cat{M}(A)$ at $\cat{T}$, $\cat{M}(A)_{\cat{T}}$, is defined as the full subcategory of $\cat{T}$-closed modules. This is an effective notion: a theorem of Gabriel and Popescu (see \cite[Theorem X.4.1]{Ste}) states that given any Grothendieck category $\cat{G}$, there is a ring $A$ and a torsion class $\cat{T}$ in $\cat{M}(A)$ such that $\cat{G}$ is equivalent to $\cat{M}(A)_{\cat{T}}$.

However, it is \emph{not} true, in general, that $\Ja_{\cat{T}}(M)$ is $\cat{T}$-closed for any $M$! Instead, one must apply $\Ja_{\cat{T}}$ twice: for any $M\in \cat{T}$, $\Ja_{\cat{T}}^{2}(M)$ is $\cat{T}$-closed, and the functor $\Ja_{\cat{T}}^{2}:\cat{M}(A) \to \cat{M}(A)_{\cat{T}}$ is left adjoint to the inclusion $\cat{M}(A)_{\cat{T}}\to \cat{M}(A)$ (\cite[Proposition IX.1.11]{Ste}). $\Ja_{\cat{T}}^{2}(A)$ does admit the structure of a ring equipped with a ring homomorphism $A\to \Ja_{\cat{T}}^{2}(A)$, and there is a canonical, fully faithful functor $\cat{M}(A)_{\cat{T}}\to \cat{M}(\Ja_{\cat{T}}^{2}(A))$. This functor is not always an equivalence; the case when it is is the context of \emph{perfect localisation} (\cite[Chapter XI]{Ste}). When $\cat{T}$ does induce a perfect localisation on $\cat{M}(A)$, $\Ja_{\cat{T}}^{2}(M) \cong \Ja_{\cat{T}}^{2}(A)\otimes_{A} M$ for all $M\in \cat{M}(A)$, the map  $A\to \Ja_{\cat{T}}^{2}(A)$ is an epimorphism in the category of rings, and $\Ja_{\cat{T}}^{2}(A)$ is flat as a right $A$-module (\cite[Proposition XI.3.4]{Ste}).

When $\cat{T}$ is stable, $\Ja_{\cat{T}} \cong \Ja_{\cat{T}}^{2}$ (\cite[Proposition IX.1.7]{Ste}). Later, we will show that the same holds for weakly stable torsion classes (Corollary \ref{c:03}).

\begin{proposition} \label{p:08}
Let $\cat{T}$ be a torsion class in $\cat{M}(A)$. Then, there is a morphism \[\lambda : \mrm{R}\Ja_{\cat{T}} \to \mrm{R}\Ga_{\cat{T}}[1]\]  of functors such that for every $M\in \cat{D}(A)$, there is a distinguished triangle \[\mrm{R}\Ga_{\cat{T}}(M)\xrightarrow{\si^{\mrm{R}}_{M}} M \xrightarrow{\tau^{\mrm{R}}_{M}} \mrm{R}\Ja_{\cat{T}}(M)\xrightarrow{\lambda_{M}} \mrm{R}\Ga_{\cat{T}}(M)[1]. \]
\end{proposition}
\begin{proof}
If $E$ is an injective module, for every $\i\in \opn{Filt}{\cat{T}}$ the sequence of $\K$-modules \[0\to \mrm{Hom}_{A}(A/\mfrak{i},E)\to E\to \mrm{Hom}_{A}(\mfrak{i},E)\to 0\] is exact. Since directed colimits are exact in $\cat{M}(A)$, \[0\to \Ga_{\cat{T}}(E)\xrightarrow{\si_{E}} E\xrightarrow{\tau_{E}} \Ja_{\cat{T}}(E)\to 0\] is exact. This in turn implies that if $I'$ is a complex of injective modules, \[0\to \Ga_{\cat{T}}(I')\xrightarrow{\si_{I'}} I'\xrightarrow{\tau_{I'}} \Ja_{\cat{T}}(I')\to 0\] is exact.

Every complex in $\cat{D}(A)$ is functorially isomorphic to a K-injective complex of injective modules. If $M\in \cat{D}(A)$, and $M\to I$ is such a resolution, the sequence of morphisms \[\mrm{R}\Ga_{\cat{T}}(M)\xrightarrow{\si^{\mrm{R}}_{M}} M \xrightarrow{\tau^{\mrm{R}}_{M}} \mrm{R}\Ja_{\cat{T}}(M)\] can be represented by the sequence 
\begin{equation} \label{e:13}
\Ga_{\cat{T}}(I)\xrightarrow{\si_{I}} I\xrightarrow{\tau_{I}} \Ja_{\cat{T}}(I).
\end{equation}
 The sequence \ref{e:13} fits into a short exact sequence \[0\to \Ga_{\cat{T}}(I)\xrightarrow{\si_{I}} I\xrightarrow{\tau_{I}} \Ja_{\cat{T}}(I)\to 0\] in $\cat{C}(A)$. Therefore, it can be functorially associated with a morphism $\Ja_{\cat{T}}(I)\to \Ga_{\cat{T}}(I)[1]$ in $\cat{D}(A)$ (\cite[Example 10.4.9]{Wei}) such that \[\Ga_{\cat{T}}(I)\to I\to \Ja_{\cat{T}}(I)\to \Ga_{\cat{T}}(I)[1]\] is a distinguished triangle in $\cat{D}(A)$. As $I$ is a K-injective resolution of $M$, the morphism $\Ja_{\cat{T}}(I)\to \Ga_{\cat{T}}(I)[1]$ corresponds to a morphism $\mrm{R}\Ja_{\cat{T}}(M)\xrightarrow{\lambda_{M}} \mrm{R}\Ga_{\cat{T}}(M)[1]$ such that the triangle
\[\mrm{R}\Ga_{\cat{T}}(M)\xrightarrow{\si^{\mrm{R}}_{M}} M \xrightarrow{\tau^{\mrm{R}}_{M}} \mrm{R}\Ja_{\cat{T}}(M)\xrightarrow{\lambda_{M}} \mrm{R}\Ga_{\cat{T}}(M)[1]\] is distinguished. 
\end{proof}

\begin{proposition} \label{p:09}
Let $\cat{T}$ be a torsion class in $\cat{M}(A)$. Then, the following are equivalent: 
\begin{enumerate}
\rmitem{i} $\cat{T}$ is weakly stable.
\rmitem{ii} $(\mathrm{R}\Ja_{\cat{T}}, \tau^{\mathrm{R}})$ is an idempotent pointed functor on $\cat{D}^{+}(A)$.
\rmitem{iii} $\mathrm{R}\Ja_{\cat{T}}(\mathrm{R}\Ga_{\cat{T}}(M))=0$ for any $M\in \cat{D}^{+}(A)$.
\rmitem{iv} $\mathrm{R}\Ga_{\cat{T}}(\mathrm{R}\Ja_{\cat{T}}(M)) = 0$ for any $M\in \cat{D}^{+}(A)$.
\end{enumerate}
If $\cat{T}$ is finite dimensional, all instances of $\cat{D}^{+}(A)$ in the above statement can be replaced with $\cat{D}(A)$.
\end{proposition}
\begin{proof}
By Proposition \ref{p:08}, for any $M\in \cat{D}(A)$ there is a distinguished triangle 

\begin{equation} \label{e:12} 
\mrm{R}\Ga_{\cat{T}}(M)\xrightarrow{\si_{M}} M \xrightarrow{\tau_{M}} \mrm{R}\Ja_{\cat{T}}(M)\xrightarrow{\lambda_{M}} \mrm{R}\Ga_{\cat{T}}(M)[1]
\end{equation} 
in $\cat{D}(A)$. In particular, given $M\in \cat{D}(A)$ we can consider the distinguished triangle corresponding to $\mathrm{R}\Ga_{\cat{T}}(M)$,
\[\mrm{R}\Ga_{\cat{T}}(\mathrm{R}\Ga_{\cat{T}}(M))\xrightarrow{\si^{\mrm{R}}_{\mathrm{R}\Ga_{\cat{T}}(M)}} \mathrm{R}\Ga_{\cat{T}}(M) \xrightarrow{\tau^{\mrm{R}}_{\mathrm{R}\Ga_{\cat{T}}(M)}} \mrm{R}\Ja_{\cat{T}}(\mathrm{R}\Ga_{\cat{T}}(M))\xrightarrow{\alpha} \mrm{R}\Ga_{\cat{T}}(\mathrm{R}\Ga_{\cat{T}}(M))[1],\]
where $\alpha = \lambda_{\mathrm{R}\Ga_{\cat{T}}(M)}$, and the distinguished triangle arising from an application of the functor $\mathrm{R}\Ga_{\cat{T}}$ to  \ref{e:12}:
\[\mrm{R}\Ga_{\cat{T}}(\mathrm{R}\Ga_{\cat{T}}(M))\xrightarrow{\mathrm{R}\Ga_{\cat{T}}(\si^{\mrm{R}}_{M})} \mathrm{R}\Ga_{\cat{T}}(M) \xrightarrow{\mathrm{R}\Ga_{\cat{T}}(\tau^{\mrm{R}}_{M})} \mrm{R}\Ga_{\cat{T}}(\mathrm{R}\Ja_{\cat{T}}(M))\xrightarrow{\alpha'} \mrm{R}\Ga_{\cat{T}}(\mathrm{R}\Ga_{\cat{T}}(M))[1],\]
where $\alpha' = \mathrm{R}\Ga_{\cat{T}}(\lambda_{M}).$ By examining the above two distinguished triangles, we see that $\si^{\mrm{R}}_{\mathrm{R}\Ga_{\cat{T}}(M)}$ is an isomorphism if and only if $\mathrm{R}\Ja_{\cat{T}}(\mrm{R}\Ga_{\cat{T}}(M)) = 0$, and that $\mathrm{R}\Ga_{\cat{T}}(\si^{\mrm{R}}_{M})$ is an isomorphism exactly when $\mathrm{R}\Ga_{\cat{T}}(\mrm{R}\Ja_{\cat{T}}(M)) = 0.$ The equivalence of (i), (iii), and (iv) is now a consequence of Proposition \ref{p:02}.

We can also consider the distinguished triangle in \ref{e:12} corresponding to $\mrm{R}\Ja_{\cat{T}}(M)$, or apply the functor $\mrm{R}\Ja_{\cat{T}}$ to the triangle in \ref{e:12}. By doing so, we see that (ii) is equivalent to (iii), and also (iv).

For finite dimensional $\cat{T}$, we use exactly the proof written above, taking into account the last line in the statement of Proposition \ref{p:02}.
\end{proof}

\begin{corollary} \label{c:03}
Let $\cat{T}$ be a weakly stable torsion class in $\cat{M}(A)$. Then, $(\Ja_{\cat{T}},\tau)$ is an idempotent copointed functor on $\cat{M}(A)$.
\end{corollary}
\begin{proof}
To begin, suppose $E$ is an injective $A$-module. By Proposition \ref{p:09}, $\Ja_{\cat{T}}(E)\cong \mrm{R}\Ja_{\cat{T}}(\Ja_{\cat{T}}(E))$, thus implying that $\Ja_{\cat{T}}(E)$ is right $\Ja_{\cat{T}}$-acyclic. 

Let $M\in \cat{M}(A)$, and let $M\to I$ be an injective resolution of $M$ with $I^{i} =0 $ for $i\leq -1$. $\Ja_{\cat{T}}(I)$ is a bounded below complex of right $\Ja_{\cat{T}}$-acyclic modules, so the map $\tau^{\mrm{R}}_{\mrm{R}\Ja_{\cat{T}}(M)}$ is represented by $\tau_{\Ja_{\cat{T}}(I)}: \Ja_{\cat{T}}(I)\to \Ja_{\cat{T}}(\Ja_{\cat{T}}(I))$. By Proposition \ref{p:09}, this map is a quasi-isomorphism so $\opn{H}^{0}(\tau_{\Ja_{\cat{T}}(I)}):\opn{H}^{0}(\Ja_{\cat{T}}(I))\to \opn{H}^{0}(\Ja_{\cat{T}}(\Ja_{\cat{T}}(I)))$ is an isomorphism. Since $\Ja_{\cat{T}}$ is left exact, we get that $\tau_{\Ja_{\cat{T}}(M)}$ is an isomorphism. 

By Lemma \ref{l:07}, $(\Ja_{\cat{T}},\tau)$ is an idempotent copointed functor on $\cat{M}(A)$.
\end{proof}

\begin{lemma} \label{l:10}
Let $\cat{T}$ be a weakly stable torsion class in $\cat{M}(A)$. The map \[\mathrm{RHom}_{A}(\tau^{\mrm{R}}_{A},\opn{id}_{\mrm{R}\Ja_{\cat{T}}(A)}): \mathrm{RHom}_{A}(\mrm{R}\Ja_{\cat{T}}(A) ,\mrm{R}\Ja_{\cat{T}}(A) ) \to  \mrm{R}\Ja_{\cat{T}}(A) \] is an isomorphism in $\cat{D}(\K)$.
\end{lemma}
\begin{proof}
Apply the triangulated functor $\mathrm{RHom}_{A}(-,\mathrm{R}\Ga_{\cat{T}}(A))$ to the distinguished triangle \ref{e:12} with $M=A$. It follows that $\mathrm{RHom}_{A}(\tau^{\mrm{R}}_{A},\opn{id}_{\mrm{R}\Ja_{\cat{T}}(A)})$ is an isomorphism if and only if \[\mathrm{RHom}_{A}(\mathrm{R}\Ga_{\cat{T}}(A), \mathrm{R}\Ja_{\cat{T}}(A)) = 0.\]

From Proposition \ref{p:04}, $(\mathrm{R}\Ga_{\cat{T}},\si^{\mrm{R}})$ is an idempotent copointed functor on $\cat{D}^{+}(A)$. Both $\mathrm{R}\Ja_{\cat{T}}(A)$ and $\mrm{R}\Ga_{\cat{T}}(A)$ belong to $\cat{D}^{+}(A)$.

 Thus, by using Lemma \ref{l:08} and considering cohomology, we see that the canonical map \[\mathrm{RHom}_{A}(\mathrm{R}\Ga_{\cat{T}}(A), \mathrm{R}\Ga_{\cat{T}}(\mathrm{R}\Ja_{\cat{T}}(A)))\to \mathrm{RHom}_{A}(\mathrm{R}\Ga_{\cat{T}}(A), \mathrm{R}\Ja_{\cat{T}}(A))\] induced by $\si^{\mrm{R}}_{\mathrm{R}\Ja_{\cat{T}}(A))}$ is an isomorphism in $\cat{D}(\K)$. Finally, Proposition \ref{p:09}  tells us that  $\mathrm{R}\Ga_{\cat{T}}(\mathrm{R}\Ja_{\cat{T}}(A))=0$.
\end{proof}

\begin{convention}
If a ring $A$ is flat as a module over $\K$, we say that $A$ is a flat ring.
\end{convention}

Suppose $A$ is a flat ring. Let $\opn{rest}: \cat{D}(A^{\mrm{en}}) \to \cat{D}(A)$ denote the restriction functor corresponding to the ring homomorphism $A\to A^{\mrm{en}}$ sending $a$ to $a\ot 1$. If $\cat{T}$ is a torsion class in $\cat{M}(A)$, there is a functor $\mathrm{R}\Ga_{\cat{T}}: \cat{D}(A^{\mrm{en}})\to \cat{D}(A^{\mrm{en}})$ such that $\mrm{R}\Ga_{\cat{T}}\circ \opn{rest} \cong \opn{rest}\circ \mrm{R}\Ga_{\cat{T}}$ (this justifies the abuse of notation). In other words, given a complex $M$ of $A^{\mrm{en}}$-modules, $\mrm{R}\Ga_{\cat{T}}(M)$ carries the natural structure of a complex of $A^{\mrm{en}}$-modules, and the morphism $\si^{\mrm{R}}_{M}$ is a morphism in $\cat{D}(A^{\mrm{en}})$. See \cite[Lemma 7.9]{VY} for details. Similar considerations apply to the functor $\mrm{R}\Ja_{\cat{T}}$ and $\tau^{\mrm{R}}$; in particular, $\mrm{R}\Ja_{\cat{T}}(A)$ is naturally an object, and $\tau^{\mrm{R}}_{A}$ a morphism, in $\cat{D}(A^{\mrm{en}})$.

\begin{lemma} \label{l:11}
Let $A$ be a flat ring, and let $\cat{T}$ be a weakly stable torsion class in $\cat{M}(A)$. The map \[\mathrm{RHom}_{A}(\tau^{\mrm{R}}_{A},\opn{id}_{\mrm{R}\Ja_{\cat{T}}(A)}): \mathrm{RHom}_{A}(\mrm{R}\Ja_{\cat{T}}(A) ,\mrm{R}\Ja_{\cat{T}}(A) ) \to  \mrm{R}\Ja_{\cat{T}}(A) \] is an isomorphism in $\cat{D}(A^{\mrm{en}})$.
\end{lemma}
\begin{proof}
In this situation, the map  \[\mathrm{RHom}_{A}(\tau^{\mrm{R}}_{A},\opn{id}_{\mrm{R}\Ja_{\cat{T}}(A)}): \mathrm{RHom}_{A}(\mrm{R}\Ja_{\cat{T}}(A) ,\mrm{R}\Ja_{\cat{T}}(A) ) \to  \mrm{R}\Ja_{\cat{T}}(A) \] is a morphism in $\cat{D}(A^{\mrm{en}})$. The proof is now identical to Lemma \ref{l:10}.
\end{proof}

\begin{theorem} \label{p:10} \label{t:03}
Let $A$ be a flat ring, and let $\cat{T}$ be a weakly stable torsion class in $\cat{M}(A)$. Then, there exists a dg ring $A_{\cat{T}}$, a morphism of dg rings $\phi_{\cat{T}}: A\to A_{\cat{T}}$, and a morphism $\psi_{\cat{T}}: A_{\cat{T}}\to \mathrm{R}\Ga_{\cat{T}}(A)[1]$ in $\cat{D}(A^{\mrm{en}})$ such that the triangle \[\mathrm{R}\Ga_{\cat{T}}(A)\xrightarrow{\si_{A}^{\mrm{R}}} A\xrightarrow{\phi_{\cat{T}}} A_{\cat{T}}\xrightarrow{\psi_{\cat{T}}} \mathrm{R}\Ga_{\cat{T}}(A)[1]\] is distinguished in $\cat{D}(A^{\mrm{en}})$.
\end{theorem}
\begin{proof}
 First, consider the distinguished triangle \[\mrm{R}\Ga_{\cat{T}}(A)\xrightarrow{\si^{\mrm{R}}_{A}} A \xrightarrow{\tau^{\mrm{R}}_{A}} \mrm{R}\Ja_{\cat{T}}(A)\xrightarrow{\lambda_{A}} \mrm{R}\Ga_{\cat{T}}(A)[1],\] viewed as a triangle in $\cat{D}(A^{\mrm{en}})$. By Lemma \ref{l:11}, the map \[\mathrm{RHom}_{A}(\tau^{\mrm{R}}_{A},\opn{id}_{\mrm{R}\Ja_{\cat{T}}(A)}): \mathrm{RHom}_{A}(\mrm{R}\Ja_{\cat{T}}(A) ,\mrm{R}\Ja_{\cat{T}}(A) ) \to  \mrm{R}\Ja_{\cat{T}}(A) \] is an isomorphism in $\cat{D}(A^{\mrm{en}})$.
 
 Let $A\xrightarrow{\mu} J$ and $\Ja_{\cat{T}}(J)\xrightarrow{\nu} I$ be K-injective resolutions of $A$ and $\Ja_{\cat{T}}(J)$, respectively, in $\cat{K}(A^{\mrm{en}})$. The map $J\xrightarrow{\tau_{J}} \Ja_{\cat{T}}(J)$ is a morphism of complexes of $A^{\mrm{en}}$-modules. 

$\opn{Hom}_{A}(I,I)$ has a canonical dg ring structure. It is easy to check that the map 
\begin{equation} \label{e:21}
\phi: A\to \mathrm{Hom}_{A}(I,I)^{\mrm{op}},
\end{equation}
where $\phi(a)(x):= xa$, is a morphism of dg rings. Morever, the $A^{\opn{en}}$-module structure induced on $\mathrm{Hom}_{A}(I,I)$ via this dg ring homomorphism is precisely the bimodule structure it inherits from the bimodule structure of $I$. 

Consider the following diagram:
\begin{displaymath}
\xymatrix{
A \ar[rr]^{\phi} \ar[d]_{\mu} & & \mathrm{Hom}_{A}(I,I) \ar[d]^{\mrm{Hom}_{A}(\nu , \opn{id}_{I})} \ar@{-->}[lldd]_{\xi} \\
J \ar[d]_{\tau_{J}} & & \mathrm{Hom}_{A}(\Ja_{\cat{T}}(J),I) \ar[d]^{\mrm{Hom}_{A}(\tau_{J} , \opn{id}_{I})} \\
\Ja_{\cat{T}}(J) \ar[d]_{\nu}& & \mathrm{Hom}_{A}(J,I) \ar[d]^{\mrm{Hom}_{A}(\mu , \opn{id}_{I})} \\
I\ar@{-}[rr]^{\cong} & & \mathrm{Hom}_{A}(A,I).
}
\end{displaymath}

Since $\mu$, $\nu$, and $\tau_{J}$ are morphisms of bimodules, the outer rectangle commutes. Since $\nu$ is an isomorphism in $\cat{D}(A^{\mrm{en}})$, we can construct a map $\xi: \mathrm{Hom}_{A}(I,I)\to \Ja_{\cat{T}}(J)$ in $\cat{D}(A^{\mrm{en}})$ making the full diagram commutative. By a judicious choice of resolutions, there is no loss of generality in assuming that $\Ja_{\cat{T}}(J) = \mrm{R}\Ja_{\cat{T}}(A)$, that $\tau_{J}\circ \mu = \tau^{\mrm{R}}_{A}$, and that $\xi$ represents the map $\mathrm{RHom}_{A}(\tau^{\mrm{R}}_{A},\opn{id}_{\mrm{R}\Ja_{\cat{T}}(A)})$. By Lemma \ref{l:11}, $\xi$ is an isomorphism in $\cat{D}(A^{\mrm{en}})$.

This allows us to construct the following diagram in $\cat{D}(A^{\mrm{en}})$:

\begin{displaymath}
\xymatrix{
\mathrm{R}\Ga_{\cat{T}}(A) \ar[d]^{\opn{id}_{\mathrm{R}\Ga_{\cat{T}}(A)}} \ar[rr]^{\si^{\mrm{R}}_{A}} & & A \ar[d]^{\opn{id}_{A}} \ar[rr]^{\phi} &  & \mathrm{Hom}_{A}(I,I) \ar[rr]^{\lambda_{A}\circ \xi} \ar[d]^{\xi} & & \mathrm{R}\Ga_{\cat{T}}(A)[1] \ar[d]^{\opn{id}_{\mathrm{R}\Ga_{\cat{T}}(A)}[1]} \\
\mathrm{R}\Ga_{\cat{T}}(A)   \ar[rr]_{\si^{\mrm{R}}_{A}} & & A \ar[rr]_{\tau^{\mrm{R}}_{A}}  & & \mrm{R}\Ja_{\cat{T}}(A) \ar[rr]_{\lambda_{A}} & & \mathrm{R}\Ga_{\cat{T}}(A)[1] }
\end{displaymath}

Each of the squares in this diagram commutes. Since the bottom row is a distinguished triangle and all the vertical morphisms are isomorphisms, the top row is a distinguished triangle as well. This proves the proposition: take $A_{\cat{T}} = \mathrm{Hom}_{A}(I,I)^{\mrm{op}}$, $\phi_{\cat{T}} = \phi$, and $\psi_{\cat{T}} = \lambda_{A}\circ \xi$.
\end{proof}

\begin{corollary} \label{c:05}
Let $A$ be a flat ring, and let $\cat{T}$ be a weakly stable torsion class in $\cat{M}(A)$. Then, there is an isomorphism  $\xi_{\cat{T}}: A_{\cat{T}} \to \mrm{R}\Ja_{\cat{T}}(A)$ in $\cat{D}(A^{\mrm{en}})$ such that the following diagram commutes:
\begin{displaymath}
\xymatrix{
A \ar[d]_{\phi_{\cat{T}}} \ar[dr]^{\tau^{\mrm{R}}_{A}} &  \\
A_{\cat{T}} \ar[r]_(0.4){\xi_{\cat{T}}} & \mrm{R}\Ja_{\cat{T}}(A).
}
\end{displaymath}
\end{corollary}
\begin{proof}
This follows from the proof of Theorem \ref{t:03}.
\end{proof}

Using Corollary \ref{c:05} and Corollary \ref{c:03}, we see that the ring $\opn{H}^{0}(A_{\cat{T}})$ is the classical localisation of $\cat{A}$ at $\cat{T}$ as constructed by Gabriel (\cite[Chapter IX]{Ste}).

\begin{remark}
The dg ring structure on $A_{\cat{T}}$ is independent of the choice of K-injective resolution of $\Ja_{\cat{T}}(J)$. More precisely, let $I$ be as in the proof of Theorem \ref{t:03}, and suppose $I'$ is another K-injective resolution of $\Ja_{\cat{T}}(J)$ in $\cat{K}(A^{\mrm{en}})$, with $\phi': A\to \mrm{Hom}_{A}(I',I')^{\mrm{op}}$ the dg ring morphism analogous to that in \ref{e:21}. Then, by using an argument very similar to that of \cite[Proposition 2.3]{PSY2}, it is possible to show that  there is a dg ring $\tilde{A}$, a morphism of dg rings $A\to \tilde{A}^{\mrm{op}}$, and dg ring quasi-isomorphisms $\tilde{A}\to A_{\cat{T}}$ and $\tilde{A}\to \mathrm{Hom}_{A}(I',I')$ such that the following diagram commutes: 
\begin{displaymath}
\xymatrix{
 & A \ar[dl]_{\phi_{\cat{T}}} \ar[d] \ar[dr]^{\phi'} & \\
A_{\cat{T}}= \mathrm{Hom}_{A}(I,I)^{\mrm{op}} & \tilde{A}^{\mrm{op}} \ar[l] \ar[r] & \mathrm{Hom}_{A}(I',I')^{\mrm{op}}. 
}
\end{displaymath}
\end{remark}

\begin{theorem} \label{p:11} 
Let $A$ be a flat ring, and $\cat{T}$ a quasi-compact, finite dimensional, weakly stable torsion class in $\cat{M}(A)$. Then, there is an isomorphism  \[\delta: A_{\cat{T}}\ot_{A}^{\mrm{L}}(-)\to \mathrm{R}\Ja_{\cat{T}}\] of triangulated functors from $\cat{D}(A)$ to itself, such that for any $M\in \cat{D}(A)$, $\delta_{M}\circ (\phi_{\cat{T}}\ot_{A}^{\mrm{L}} \opn{id}_{M}) = \tau^{\mrm{R}}_{M}$.
\end{theorem}
\begin{proof}
Suppose $M\in \cat{D}(A)$. There are two distinguished triangles that we can associate with $M$. The first is \[\mrm{R}\Ga_{\cat{T}}(A)\ot_{A}^{\mrm{L}} M \xrightarrow{\si^{\mrm{R}}_{A}\otimes^{\mrm{L}}_{A} \opn{id}_{M}} M \xrightarrow{\phi_{\cat{T}}\ot_{A}^{\mrm{L}} \opn{id}_{M}} A_{\cat{T}} \ot_{A}^{\mrm{L}} M\xrightarrow{\psi_{\cat{T}}\ot_{A}^{\mrm{L}} \opn{id}_{M}} (\mrm{R}\Ga_{\cat{T}}(A)\ot_{A}^{\mrm{L}} M)[1].\] The second is \[\mrm{R}\Ga_{\cat{T}}(M)\xrightarrow{\si^{\mrm{R}}_{M}} M \xrightarrow{\tau^{\mrm{R}}_{M}} \mrm{R}\Ja_{\cat{T}}(M)\xrightarrow{\lambda_{M}} \mrm{R}\Ga_{\cat{T}}(M)[1].\]

From \cite[Theorem 7.12]{VY}, we know there is an isomorphism \[\gamma: \mathrm{R}\Ga_{\cat{T}}(A)\ot_{A}^{\mrm{L}}(-)\to \mrm{R}\Ga_{\cat{T}}\]  of triangulated functors from $\cat{D}(A)$ to itself such that for any $M\in \cat{D}(A)$, $\si^{\mrm{R}}_{M}\circ \gamma_{M} = \si^{\mrm{R}}_{A}\ot_{A}^{\mrm{L}} \opn{id}_{M}.$ Consider the diagram, 

\begin{displaymath}
\xymatrix{
\mrm{R}\Ga_{\cat{T}}(A)\ot_{A}^{\mrm{L}} M \ar[rr]^(.60){\si^{\mrm{R}}_{A}\otimes_{A}^{\mrm{L}} \opn{id}_{M}} \ar[d]^{\gamma_{M}}  & & M \ar[rr]^(.45){\phi_{\cat{T}}\ot_{A}^{\mrm{L}} \opn{id}_{M}} \ar[d]^{\opn{id}_{M}} & &  A_{\cat{T}} \ot_{A}^{\mrm{L}} M \ar[r]^(.4){\alpha}  \ar@{-->}[d]^{\delta_{M}} & (\mrm{R}\Ga_{\cat{T}}(A)\ot_{A}^{\mrm{L}} M)[1] \ar[d]^{\gamma_{M}[1]} \\
\mrm{R}\Ga_{\cat{T}}(M) \ar[rr]^{\si^{\mrm{R}}_{M}}  & &  M \ar[rr]^{\tau^{\mrm{R}}_{M}} &  & \mrm{R}\Ja_{\cat{T}}(M) \ar[r]^{\lambda_{M}}   & \mrm{R}\Ga_{\cat{T}}(M)[1],
}
\end{displaymath}

where $\alpha = \psi_{\cat{T}}\ot_{A}^{\mrm{L}} \opn{id}_{M}$.
The leftmost square in the above diagram is commutative, and all three solid vertical arrows are isomorphisms. The axioms for a triangulated category imply that there is a map $\delta_{M}: A_{\cat{T}}\ot_{A}^{\mrm{L}} M\to \mrm{R}\Ja_{\cat{T}}(M)$ which makes the above diagram commute. By itself, this is not enough: while the axioms for a triangulated category guarantee the existence of an isomorphism they say nothing about uniqueness or functoriality, and  we want the maps $\delta_{M}$, as $M$ varies over $\cat{D}(A)$, to assemble together into a morphism of functors. 

In this particular case, however, the choice of $\delta_{M}$ \emph{is} unique. The argument is not dissimilar to the proof of Lemma \ref{l:03}, but we give the details. Suppose $M, N \in \cat{D}(A)$. By Lemma \ref{l:08} and Proposition \ref{p:09}, $\mathrm{Hom}_{\cat{D}(A)}(\mrm{R}\Ga_{\cat{T}}(M), \mrm{R}\Ja_{\cat{T}}(N)) = 0$. Since $\gamma_{M}$ is an isomorphism, this implies that $\mathrm{Hom}_{\cat{D}(A)}(\mrm{R}\Ga_{\cat{T}}(A)\otimes_{A}^{\mrm{L}} M, \mrm{R}\Ja_{\cat{T}}(N)) =0 $.

Applying the cohomological functor $\mathrm{Hom}_{\cat{D}(A)}(-, \mrm{R}\Ja_{\cat{T}}(N))$ to the distinguished triangle  \[\mrm{R}\Ga_{\cat{T}}(A)\ot_{A}^{\mrm{L}} M \xrightarrow{\si^{\mrm{R}}_{A}\otimes^{\mrm{L}}_{A} \opn{id}_{M}} M \xrightarrow{\phi_{\cat{T}}\ot^{\mrm{L}}_{A} \opn{id}_{M}} A_{\cat{T}} \ot_{A}^{\mrm{L}} M\xrightarrow{\psi_{\cat{T}}\ot_{A}^{\mrm{L}} \opn{id}_{M}} (\mrm{R}\Ga_{\cat{T}}(A)\ot_{A}^{\mrm{L}} M)[1],\] we conclude that the map $\mathrm{Hom}_{\cat{D}(A)}(\phi_{\cat{T}}\ot^{\mrm{L}}_{A} \opn{id}_{M}, \opn{id}_{\mrm{R}\Ja_{\cat{T}}(N)}):$ \[\mathrm{Hom}_{\cat{D}(A)}(A_{\cat{T}} \ot_{A}^{\mrm{L}} M, \mrm{R}\Ja_{\cat{T}}(N))\to \mathrm{Hom}_{\cat{D}(A)}(M, \mrm{R}\Ja_{\cat{T}}(N))\] is an isomorphism.
Taking $M = N$, this proves the uniqueness of $\delta_{M}$.

Let $f:M\to N$ be a morphism in $\cat{D}(A)$. We want to show that \[\mrm{R}\Ja_{\cat{T}}(f)\circ \delta_{M} = \delta_{N} \circ (\opn{id}_{A_{\cat{T}}}\ot_{A}^{\mrm{L}} f).\] It is enough to show that both sides are equal after precomposing by $\phi_{\cat{T}}\ot_{A}^{\mrm{L}} \opn{id}_{M}$. We have just proved that for any $M\in \cat{D}(A)$, \[\delta_{M}\circ (\phi_{\cat{T}}\ot_{A}^{\mrm{L}} \opn{id}_{M}) = \tau^{\mrm{R}}_{M}.\] Therefore \[\mrm{R}\Ja_{\cat{T}}(f)\circ \delta_{M}  \circ (\phi_{\cat{T}}\ot_{A}^{\mrm{L}} \opn{id}_{M}) = \mrm{R}\Ja_{\cat{T}}(f) \circ \tau^{\mrm{R}}_{M}.\] On the other hand, \[ \delta_{N} \circ (\opn{id}_{A_{\cat{T}}}\ot_{A}^{\mrm{L}} f) \circ (\phi_{\cat{T}}\ot_{A}^{\mrm{L}} \opn{id}_{M}) = \delta_{N} \circ (\phi_{\cat{T}}\ot_{A}^{\mrm{L}}  f )= \tau^{\mrm{R}}_{N} \circ f.\] Finally,  $\mrm{R}\Ja_{\cat{T}}(f) \circ \tau^{\mrm{R}}_{M} = \tau^{\mrm{R}}_{N} \circ f.$
A similar argument implies that $\delta$ is a morphism of triangulated functors.
\end{proof}

Recall that a morphism of dg rings $\phi:A\to B$ is called a \emph{homological epimorphism} if the canonical multiplication map $B\ot_{A}^{\mrm{L}} B\to B$ is an isomorphism in $\cat{D}(\K)$. This is equivalent to either of the maps $\phi\ot_{A}^{\mrm{L}} \opn{id}_{B}: B\to B\ot_{A}^{\mrm{L}} B$, $\opn{id}_{B}\ot_{A}^{\mrm{L}} \phi: B\to B\ot_{A}^{\mrm{L}} B$ being isomorphims in $\cat{D}(\K)$, and is also equivalent to the restriction functor $\opn{rest}_{\phi}:\cat{D}(B)\to \cat{D}(A)$ being fully faithful. 

\begin{corollary} \label{c:01}
Let $A$ be a flat ring, and $\cat{T}$ a quasi-compact, finite dimensional, weakly stable torsion class in $\cat{M}(A)$. Then, the morphism of dg rings $\phi_{\cat{T}}: A\to A_{\cat{T}}$ is a homological epimorphism.
\end{corollary}
\begin{proof}
This is a consequence of Proposition \ref{p:04} and Theorem \ref{p:11}.
\end{proof}

\begin{corollary} \label{c:02}
Let $A$ be a flat ring, and $\cat{T}$ a quasi-compact, finite dimensional, weakly stable torsion class in $\cat{M}(A)$. Then, an object lies in the essential image of the functor $\opn{rest}_{\phi_{\cat{T}}}:\cat{D}(A_{\cat{T}})\to \cat{D}(A)$ if and only if $\mathrm{R}\Ga_{\cat{T}}(M)=0$.
\end{corollary}
\begin{proof}
Since $\phi_{\cat{T}}:A\to A_{\cat{T}}$ is a homological epiomorphism, $M\in \cat{D}(A)$ lies in the essential image of $\opn{rest}_{\phi_{\cat{T}}}$ if and only if $\phi_{\cat{T}}\ot_{A}^{\mrm{L}} \opn{id}_{M}$ is an isomorphism. By Theorem \ref{p:11}, this happens if and only if $\tau^{\mrm{R}}_{M}$ is an isomorphism which in turn, by Proposition \ref{p:08}, occurs precisely when $\mrm{R}\Ga_{\cat{T}}(M) = 0$.
\end{proof}

\begin{remark}
It is possible to work without the restriction that $A$ is a flat ring. In this case, however, we must replace the category $\cat{D}(A^{\mrm{en}})$ with $\cat{D}(A\ot_{\K}^{\mrm{L}} A^{\mrm{op}})$ (see Remark \ref{r:01}). With this modification, all of the results in this section continue to hold, but the arguments required to justify them are more technically involved.
\end{remark}

\begin{remark}
Let $\a$ be a weakly proregular ideal in a commutative ring $A$, with associated torsion class $\cat{T}_{\a}$. In this case, there is an explicit description of the dg ring $A_{\a}:= A_{\cat{T}_{\a}}$: see \cite[Section 7]{PSY1}.
\end{remark}

\begin{remark}
When $\cat{T}$ is quasi-compact, finite dimensional, and weakly stable, the functor $\mrm{R}\Ga_{\cat{T}}$ induces a recollement on the category $\cat{D}(A)$ (\cite[Proposition 4.13.1]{Kr}). In \cite[Theorem 4]{NS}, Nicol\'{a}s and Saor\'{i}n show that recollements on a dg category $\cat{A}$ are parametrised by homological epimorphisms of dg categories $\cat{A}\to \cat{B}$. The morphism $\phi_{\cat{T}}:A\to \cat{A}_{\cat{T}}$ fits into their framework, but can be arrived at here by more direct means. 
\end{remark}

\section{Completion: The case of a single regular element} \label{s:06}

To begin, suppose $A$ is a flat ring, and $\cat{T}$ is a quasi-compact, finite dimensional, weakly stable torsion class in $\cat{M}(A)$. In \cite[Theorem 4.12]{VY}, the author and Yekutieli prove that there is an isomorphism \[\gamma: \mathrm{R}\Ga_{\cat{T}}(A)\ot_{A}^{\mrm{L}} (-) \to \mrm{R}\Ga_{\cat{T}}\] of triangulated functors from $\cat{D}(A)$ to itself. This in turn implies that the functor $\mrm{R}\Ga_{\cat{T}}$ has a left adjoint, $\mathrm{G}_{\cat{T}}: = \mrm{RHom}_{A}(\mrm{R}\Ga_{\cat{T}}(A),-): \cat{D}(A)\to \cat{D}(A).$

Now suppose $A$ is a commutative ring, and $\a$ a weakly proregular ideal in $A$. The torsion class associated to $\a$, $\cat{T}_{\a}$ is quasi-compact and finite dimensional (\cite[Corollary 4.21]{VY}). In this case, it is possible to describe the functor $\mathrm{G}_{\a}:= \mrm{RHom}_{A}(\mrm{R}\Ga_{\a}(A),-): \cat{D}(A)\to \cat{D}(A)$ in a more explicit manner.

If $\a$ is an ideal in a commutative ring $A$, let \[\Lambda_{\a}: \cat{M}(A)\to \cat{M}(A)\] denote the $\a$-adic completion functor. Explicitly, if $M\in \cat{M}(A)$, \[ \Lambda_{\a}(M):= \varprojlim_{n\in \mathbb{Z}^{+}} M/\a^{n}M = \varprojlim_{n\in \mathbb{Z}^{+}} (A/\a^{n}\ot_{A} M).\]

\begin{theorem}\cite[Theorem 6.11]{PSY1} Let $\a$ be a weakly proregular ideal in a commutative ring $A$. Then, $\mrm{L}\Lambda_{\a}:\cat{D}(A)\to \cat{D}(A)$ is right adjoint to $\mrm{R}\Gamma_{\a}:\cat{D}(A)\to \cat{D}(A).$
\end{theorem}

\begin{example}
Let $A$ be an integral domain with field of fractions $F$ such that the projective dimension of $F$ over $A$ is greater than or equal to two. Let $S:= A\setminus \{0\}$. The torsion class $\cat{T}_{S}$ is weakly stable (indeed, it is stable). Thus, by Theorem \ref{t:02}, $\cat{T}_{S}$ is also quasi-compact and finite dimensional, and $\mrm{R}\Ga_{S} \cong \opn{K}(A;S)\ot^{\mrm{L}}_{A} (-)$, which implies that  $\mathrm{G}_{\cat{T}}\cong \mrm{RHom}_{A}(\opn{K}(A;S),-).$ Since the projective dimension of $F$ is greater than one, there exists $M\in \cat{M}(A)$ such that $\opn{H}^{1}(\mrm{RHom}_{A}(\opn{K}(A;S),M))\neq 0$. Thus $\mathrm{G}_{\cat{T}}$ cannot be equal to $\mrm{L}F$ for any functor $F:\cat{M}(A)\to \cat{M}(A)$.
\end{example}

\begin{question}
Let $A$ be a ring, and $\cat{T}$ a quasi-compact, finite dimensional, weakly stable torsion class. When does there exist a functor $\Lambda_{\cat{T}}:\cat{M}(A)\to \cat{M}(A)$ such that $\mathrm{G}_{\cat{T}}\cong \mrm{L}\Lambda_{\cat{T}}$? As the above example shows, this cannot always happen. It is unclear whether the projective dimension of $\mrm{R}\Ga_{\cat{T}}(A)$ is the only obstacle here. 
\end{question}

There is a noncommutative situation where we can make a positive claim in this regard. Let $A$ be a ring (not necessarily commutative, or flat). Recall that an element $s\in A$ is \textit{normal} if $As=sA$, and \textit{regular} if it has no left or right zero-divisors. 

When $s$ is normal and regular, the set $S=\{{s^{n}}\}_{n\in \mathbb{N}}$ is a left denominator set. We will denote the localisation of $A$ at $S$ by $A_{s}$. Denote $\cat{T}_{s}:=\cat{T}_{S}$ and $\Ga_{s}:=\Ga_{S}$. As in \S \ref{s:03}, there is the morphism of functors $\si^{\mrm{R}}:\mrm{R}\Ga_{s}\to \opn{id}_{\cat{D}(A)}$.

\begin{lemma} \label{l:04}
Let $s$ be a normal and regular element in a ring $A$. The torsion class $\cat{T}_{s}$ is quasi-compact, weakly stable, and has cohomological dimension less than or equal to one.
\end{lemma}
\begin{proof}
Since $s$ is regular, for any $n\in \mathbb{N}$ the two term complex
\begin{equation} \label{e:17}
\ldots \to 0\to A\xrightarrow{\cdot s^{n}} A\to 0\to \ldots,
\end{equation}
with $A$ sitting in degrees $-1$ and $0$, is a projective resolution of $A/As^{n}$.
This implies that $\mrm{Ext}^{i}_{A}(A/As^{n},M) = 0$ for all $M\in \cat{M}(A),$ $i\geq 2$, and $n\in \mathbb{N}.$ Since 
\begin{equation} \label{e:14}
\mrm{R}^{i}\Ga_{s}\cong \varinjlim_{n\in \mathbb{N}} \mathrm{Ext}^{i}_{A}(A/As^{n}-),
\end{equation} 
for all $i\geq 0$, this in turn implies that the  cohomological dimension of $\Gamma_{s}$ is less than or equal to one. By Lemma \ref{l:02}, $\cat{T}_{s}$ is weakly stable.  Using the resolution in \ref{e:17}, the formula in \ref{e:14} also implies that $\cat{T}_{s}$ is quasi-compact (alternatively, use Theorem \ref{t:02}).
\end{proof}

Given a normal element $s$ in $A$, $As=sA$ is the ideal generated by $s$, which we denote by $(s)$. $(s)^{n}=(s^{n})=As^{n}=s^{n}A$. Let \[\Lambda_{s}: \cat{M}(A)\to \cat{M}(A)\] denote the completion functor at this ideal. Explicitly, if $M\in \cat{M}(A)$, \[ \Lambda_{s}(M):= \varprojlim_{n\in \mathbb{Z}^{+}} M/s^{n}M = \varprojlim_{n\in \mathbb{Z}^{+}} (A/(s)^{n}\ot_{A} M).\]

If $s$ is a normal and regular element in a ring $A$, and $S=\{{s^{n}}\}_{n\in \mathbb{N}}$ the associated left denominator set, we denote the complex $\opn{K}(A;S)$ as in Definition \ref{d:04} by $\opn{K}(A;s)$. We remind the reader that \[\opn{K}(A;s) = \ldots \to 0\to A\to A_{s}\to 0 \to \ldots\] is a complex in $\cat{C}(A^\mrm{en})$,  where $A$ and $A_{s}$ sit in degrees $0$ and $1$ respectively. There is a morphism of complexes $e_{s}: \opn{K}(A;s)\to A.$

\begin{lemma} \label{l:12}
Let $s$ be a normal and regular element in a ring $A$. Then, there is an isomorphism \[\nu_{s}: \mrm{R}\Ga_{s} \to \opn{K}(A;s)\ot_{A}^{\mrm{L}}(-)\] of triangulated functors from $\cat{D}(A)$ to itself, such that $(e_{s}\ot_{A}^{\mrm{L}} \opn{id}_{(-)})\circ \nu_{s} = \si^{R}.$
\end{lemma}
\begin{proof}
By Lemma \ref{l:04}, $\cat{T}_{s}$ is weakly stable. We now apply Theorem \ref{t:02}.
\end{proof}

\begin{theorem} \label{t:01}
Let $s$ be a normal and regular element in a ring $A$. Then, there is an isomorphism \[\omega_{s}: \mrm{RHom}_{A}(\cat{K}(A;s), -)\to \mrm{L}\Lambda_{s}\] of triangulated functors from $\cat{D}(A)$ to itself.
\end{theorem}

\begin{corollary} \label{c:04}
Let $s$ be a normal and regular element in a ring $A$. Then, $\mrm{L}\Lambda_{s}:\cat{D}(A)\to \cat{D}(A)$ is right adjoint to $\mrm{R}\Ga_{s}:\cat{D}(A)\to \cat{D}(A).$ \qed
\end{corollary}

The proof of Theorem \ref{t:01} is an adaptation of the proof of \cite[Corollary 4.25]{PSY1}. However, since there are some noncommutative subtleties involved, we give the argument in full.

Let $s$ be a normal and regular element in a ring $A$.  Given $a\in A$, there exists a unique element $\phi(a)\in A$ such that $as=s\phi(a)$. This correspondence assembles together to define a ring automorphism $\phi: A\to A$. 

\begin{definition} \label{d:11}
Let $A$ be a ring. If $M\in \cat{M}(A)$, let ${}^{\phi^{i}}M$ denote the $A$-module with the action of $A$ twisted by $\phi^{i}$. Explicitly, ${}^{\phi^{i}}M$ has the same underlying $\K$-module as $M$, but with the following action of $A$: for $a\in A$ and $m\in {}^{\phi^{i}}M$, $a \cdot m := \phi^{i}(a)m$.
\end{definition}

For each $i\in \mathbb{N}$, $As^{-i}$, is an $A^{\mrm{en}}$-submodule of $A_{s}$. If we view $As^{-i}$ as a rank $1$ free left $A$-module with basis $s^{-i}$, $A$ acts on the right by \[(as^{-i})b = a\phi^{i}(b)s^{-1}.\]

Consider $\bigoplus_{i\in \mathbb{N}}As^{-i} \in \cat{M}(A^{\mrm{en}}).$ This is a free left $A$-module of countably infinite rank. Let $e_{i}:= (s^{-i})_{i}$. The set $\{e_{i}\}_{i\in \mathbb{N}}$ is a left $A$-module basis for $\bigoplus_{i\in \mathbb{N}}As^{-i}.$

Define the following complex in $\cat{C}(A^{\mrm{en}})$, concentrated in degrees $0$ and $1$:
\[ \mrm{Tel'}(A;s):= \bigl( \cdots \to 0 \to \bigoplus_{i\in \mathbb{N}}As^{-i} \xrightarrow{\partial} \bigoplus_{i\in \mathbb{N}}As^{-i} \to 0 \to \cdots \bigr), \] with differential $\partial$ given by defining \[\partial(e_{i})= e_{i} - se_{i+1},\] and then extending linearly as a left $A$-module map.  Since \[\partial(a\phi^{i}(b)e_{i}) = a\phi^{i}(b)e_{i} - a\phi^{i}(b)se_{i+1} =a\phi^{i}(b)e_{i} - as\phi^{i+1}(b)e_{i+1} = ae_{i}b -ase_{i+1}b,\]

the map $\partial$ is a morphism of $A^{\mrm{en}}$-modules. There is a map \[\beta_{s}: \mrm{Tel'}(A;s)\to A_{s}[-1],\] defined by sending $e_{i}$ to $s^{-i}$ and extending linearly. $\beta_{s}$ is a morphism in $\cat{C}(A^{\mrm{en}})$.

\begin{lemma} \label{l:05}
$\beta_{s}$ is a quasi-isomorphism.
\end{lemma}
\begin{proof}
If  $a= a_{0}e_{0} +  a_{1}e_{1} +  \ldots +  a_{n}e_{n} \in \opn{Tel}'(A,s)^{0},$ \[\partial(a) = a_{0}e_{0} + (a_{1} - a_{0}s)e_{1} + (a_{2} - a_{1}s)e_{2}
+ \ldots - a_{n}se_{n+1}.\] From this formula, we see that $\partial$ is injective, and thus that $H^{0}(\opn{Tel'}(A;s))=0$.

It is clear that $\beta_{s}$ is surjective. Let $a=a_{0}e_{0} +  a_{1}e_{1} + \ldots + a_{n}e_{n} \in \opn{Tel}'(A,s)^{1},$ and suppose $\beta_{s}(a)=0$. Thus \[a_{0} + a_{1}s^{-1} + \ldots + a_{n}s^{-n} = 0.\]

Define $b:= a_{0}e_{0} + (a_{1}+a_{0}s)e_{1} +  (a_{2} + a_{1}s + a_{0}s^{2})e_{2} +  \ldots +  (a_{n-1} + a_{n-2}s + \ldots + a_{0}s^{n-1})e_{n-1} \in \opn{Tel}'(A,s)^{0}.$ Then, $\partial(b)=a$.
\end{proof}

There is a morphism  \[\theta_{s}: A[-1]\to \mrm{Tel'}(A;s),\] in $\cat{C}(A^{\mrm{en}}),$ defined by sending $1$ to $e_{0}$. $\beta_{s}\circ \theta_{s}: A[-1]\to A_{s}[-1]$ is the translation of the canonical localisation map from $A$ to $A_{s}$.

Let $\mrm{Tel}(A;s) \in \cat{C}(A^{\mrm{en}})$ denote the mapping cone of $\theta_{s}$. $\mrm{Tel}(A;s)$ has an explicit description which is similar to that of $\mrm{Tel'}(A,s).$ Consider the $A^{\mrm{en}}$-module $A\oplus (\bigoplus_{i\in \mathbb{N}}As^{-i})$, with left $A$-module basis $\{e_{-1}\}\cup \{e_{i}\}_{i\in \mathbb{N}},$ where $e_{-1}$ corresponds to $1\in A$. 

Then, \[ \mrm{Tel}(A;s):= \bigl( \cdots \to 0 \to A\oplus (\bigoplus_{i\in \mathbb{N}}As^{-i}) \xrightarrow{\partial} \bigoplus_{i\in \mathbb{N}}As^{-i} \to 0 \to \cdots \bigr),\] with differential $\partial$ given by \[\partial(e_{-1})= e_{0},\] and \[\partial(e_{i})= e_{i} - se_{i+1},\] for $i\geq 0$.

Observe that  $\opn{K}(A;s)$ is the mapping cone of the map $A[-1]\to A_{s}[-1].$

\begin{lemma} \label{l:06}
There is an isomorphism \[\gamma_{s}: \opn{Tel}(A;s)\to \opn{K}(A;s)\] in $\cat{D}(A^{\mrm{en}})$.
\end{lemma}
\begin{proof}
There is the following diagram in $\cat{D}(A^{\mrm{en}})$:

\begin{displaymath}
\xymatrix{
A[-1] \ar[r]^{\theta_{s}} \ar[d]^{\opn{id}_{A[-1]}} & \opn{Tel'}(A,s) \ar[r] \ar[d]^{\beta_{s}} & \opn{Tel}(A,s) \ar[r] \ar@{-->}[d]^{\gamma_{s}}  & A \ar[d]^{\opn{id}_{A}} \\
A[-1] \ar[r] & A_{s}[-1] \ar[r] & \opn{K}(A;s)  \ar[r] & A, 
}
\end{displaymath}

where both rows are distinguished triangles, the leftmost square commutes, and all solid vertical maps are isomorphisms. The axioms for a triangulated category now imply the existence of an isomorphism $\gamma_{s}: \opn{Tel}(A;s)\to \opn{K}(A;s).$
\end{proof}

\begin{proof}[Proof of Theorem \ref{t:01}]

$\opn{Tel}(A;s)$ is a bounded complex of $A^{\mrm{en}}$-modules; as a complex of left $A$-modules, it is termwise projective. By Lemma \ref{l:06}, it can therefore be used to compute $\mrm{RHom}_{A}(\opn{K}(A;s),-)$. 

Let $M\in \cat{M}(A)$. After adjusting indices, the complex $\mrm{Hom}_{A}(\opn{Tel}(A;s),M)$ is given by the following complex of $A$-modules, concentrated in degrees $-1$ and $0$:

\[\bigl( \cdots \to 0 \to \prod_{i\in \mathbb{N}} {}^{\phi^{i}}M \xrightarrow{\partial} M\times \prod_{i\in \mathbb{N}} {}^{\phi^{i}} M \to 0 \to \cdots \bigr),\] with differential \[\partial((m_{0},m_{1},m_{2},\ldots)) = (m_{0}, m_{0} - sm_{1}, m_{1}-sm_{2}, \dots). \]

Define a morphism \[\omega_{M}: M\times \prod_{i\in \mathbb{N}} {}^{\phi^{i}} M  \to \varprojlim_{n\in \mathbb{Z}^{+}} M/s^{n}M \] by \[\omega_{M}((m,m_{0},m_{1},m_{2},\ldots))= (m_{0}-m, sm_{1} + m_{0} - m, s^{2}m_{2} + sm_{1} + m_{0} - m, \ldots). \]

There is a commutative diagram in $\cat{C}(A)$:

\begin{displaymath}
\xymatrix{
\ldots \ar[r] & \prod_{i\in \mathbb{N}} {}^{\phi^{i}}M \ar[r]^(.43){\partial} \ar[d] & M\times \prod_{i\in \mathbb{N}} {}^{\phi^{i}} M \ar[r] \ar[d]^{\omega_{M}} & \ldots\\
\ldots \ar[r] & 0\ar[r] & \varprojlim_{n\in \mathbb{Z}^{+}} M/s^{n}M \ar[r] & \ldots.
}
\end{displaymath}

The morphism $\omega_{M}$ extends naturally to a morphism \[\omega: \mrm{Hom}_{A}(\opn{Tel}(A;s),-)\to \Lambda_{s}(-)\] of triangulated functors from $\cat{K}(A)$ to itself.

Suppose $L$ is a flat $A$-module. We claim that $\omega_{L}$ is a quasi-isomorphism. 

For flat $L$, the $\K$-linear map $s \cdot: L\to L$ is injective. This implies that  \[\opn{H}^{-1}(\mrm{Hom}_{A}(\opn{Tel}(A;s),L))=0.\] Thus, all we need to show is that  $\opn{im}(\partial)=\ker(\omega_{L}).$  To see this, suppose \[\omega_{L}((l,l_{0},l_{1},\ldots))=0.\] We wish to find \[(x_{0}, x_{1},x_{2},\ldots)\in \prod_{i\in \mathbb{N}} {}^{\phi^{i}}M\] such that  \[(x_{0}, x_{0} -sx_{1}, x_{1} -sx_{2},\ldots) = (l,l_{0},l_{1},\ldots).\] 

To do this, first define $x_{0} = l$. Suppose $k\geq 1$. Since \[s^{k-1}l_{k-1} + s^{k-2}l_{k-2} + \ldots + sl_{1} + l_{0} -l\in s^{k}L,\] there is a unique $x_{k}\in L$ such that \[s^{k-1}l_{k-1} + s^{k-2}l_{k-2} + \ldots + sl_{1} + l_{0} -l = -s^{k}x_{k}.\]  

It follows that \[s^{k-1}l_{k-1} - s^{k-1}x_{k-1} = -s^{k}x_{k}.\] Since multiplication by $s$ is an injective map from $M$ to itself, \[l_{k-1}-x_{k-1}= -sx_{k},\] and thus \[l_{k-1} = x_{k-1} - sx_{k}.\] 

We have shown that $\omega_{L}$ is a quasi-isomorphism for any flat module $L$. However, since $\opn{Tel}(A;s)$ is a bounded complex of projective modules, this implies that $\omega_{L}$ is a quasi isomorphism for any  complex of flat $A$-modules (\cite[Proposition 1.9]{PSY1}), and thus in particular for any K-projective complex of projective modules. Since every complex in $\cat{D}(A)$ is functorially isomorphic to such a complex, $\omega$ induces an isomorphism \[\omega_{s}: \mrm{RHom}_{A}(\opn{K}(A;s),-)\to \mrm{L}\Lambda_{s}\] of triangulated functors from $\cat{D}(A)$ to itself.
\end{proof}


\begin{thebibliography}{PSY2}

\bibitem[AJL]{AJL} L. Alonso, A. Jeremias and J. Lipman, \textit{Local homology and cohomology on schemes},  Ann.~Sci.~\'{E}c.~Norm.~Sup\'{e}r., {\bf 30}, (1997), 1-39.
  Correction, availabe online.
  
 \bibitem[APT]{APT1} M. Auslander, M.I. Platzeck, G, Todorov, \textit{Homological theory of idempotent ideals}, Trans.~Amer.~Math.~Soc., \textbf{332} (2), (1992) 667-692.

\bibitem[AILN]{AILN} L.L. Avramov, S. B. Iyengar, J. Lipman, S. Nayak, \textit{Reduction of derived Hochschild functors over commutative algebras and schemes}, Adv.~Math., \textbf{223} (2), (2010), 735-772.


\bibitem[GJ]{GJ} K.R. Goodearl and D.A. Jordan, \textit{Localizations of essential extensions}, Proc.~Edinb.~Math.~Soc., \textbf{31}, (1988), 243-247.

\bibitem[GW]{GW} K.R. Goodearl and R.B. Warfield, Jr., An Introduction to Noncommutative Noetherian Rings, Second Edition, London Mathematical Society Student Texts  {\bf 61}, Cambridge University Press 2004. 

\bibitem[GM]{GM} J.P.C. Greenlees and J.P. May, \textit{Derived functors of I-adic completion and local homology}, J.~Algebra, {\bf 149}, (1992), 438-453.


\bibitem[RD]{RD} R. Hartshorne, Residues and Duality, Lecture Notes in Mathematics \textbf{20},  Springer-Verlag, Berlin, 1966.


\bibitem[LC]{LC} R. Hartshorne, Local cohomology: a seminar given by A. Grothendieck, Lecture Notes in Mathematics {\bf41}, Springer 1967.


\bibitem[KS]{KS} M. Kashiwara and P. Schapira, Categories and Sheaves, Grundlehren der mathematischen Wissenschaften {\bf 332}, Springer 2006. 



\bibitem[Kr]{Kr} H. Krause, \textit{Localization theory for triangulated categories},  Triangulated categories, London Mathematical Society Lecture Note Series {\bf375}, Cambridge University 
Press 2010,  161-253.

\bibitem[Ma]{Ma} E. Matlis, \textit{The higher properties of R-sequences}, J.~Algebra, {\bf 50}, (1978), 77-112.


\bibitem[PSY1]{PSY1} M. Porta, L. Shaul and A. Yekutieli, \textit{On the homology of completion and torsion}, Algebr.~Represent.~Theory, {\bf 17}, (2014), 31-67.

\bibitem[PSY2]{PSY2} M. Porta, L. Shaul and A. Yekutieli, \textit{Completion by derived double centralizer}, Algebr.~Represent.~Theory, {\bf 17}, (2014), 481-494. 

\bibitem[Po]{Po} L. Positselski, \textit{Dedualizing complexes and MGM duality}, arXiv:1503.05523.



\bibitem[NS]{NS} P. Nicol\'{a}s and M. Saor\'{i}n, \textit{Parametrizing recollement data for triangulated categories}, J.~Algebra, \textbf{322}, (2009), 1220-1250.

\bibitem[Sch]{Sch} P. Schenzel, \textit{Proregular sequences, local cohomology, and completion}, Math.~Scand., {\bf 92}, (2003), 181-180.


\bibitem[Ste]{Ste} B. Stenstr{\"{o}}m, Rings of quotients, Grundlehren der mathematischen Wissenschaften \textbf{217}, Springer-Verlag 
1975.


\bibitem[VY1]{VY} R. Vyas and A. Yekutieli, \textit{Weak stability, weak progregularity, and the noncommutative MGM equivalence}, arXiv:1608.03543, 44 pages. 

\bibitem[VY2]{VY2} R. Vyas and A. Yekutieli, \textit{Dualizing complexes in the noncommutative arithmetic context}, in preparation. 

\bibitem[Wei]{Wei} C. A. Weibel, An introduction to homological algebra, Cambridge studies in advanced mathematics \textbf{38}, Cambridge University Press 1994.

\bibitem[Ye3]{Ye3} A. Yekutieli, \textit{Derived categories of bimodules}, in preparation. 



\end{thebibliography}
\end{document}